\documentclass[11pt]{amsart}
\usepackage{etex}
\usepackage{amsmath, amscd, amsthm, amssymb, graphics, mathrsfs, setspace,fancyhdr,geometry,tabularx,shapepar, hyperref, xcolor, tikz, cite, booktabs, dsfont,amsfonts, marvosym, amsbsy, bm, tikz, array,}
\usetikzlibrary{matrix,arrows,backgrounds}
\usepackage[all]{xy}
\hypersetup{colorlinks=false}
\theoremstyle{plain}
\newtheorem{thm}{Theorem}[section]

\newtheorem{lem}[thm]{Lemma}
\newtheorem{cor}[thm]{Corollary}
\newtheorem{prop}[thm]{Proposition}

\newtheorem{question}[thm]{Question}

\newtheorem{conj}[thm]{Conjecture}
\newtheorem*{thm*}{Theorem}
\newtheorem*{problem*}{Problem}

\theoremstyle{definition}
\newtheorem{defn}[thm]{Definition}

\newtheorem{rem}[thm]{Remark}

\newtheorem{ex}[thm]{Example}

\newenvironment{sproof}{%
  \proof}{\endproof}

\numberwithin{equation}{section}

\newcommand{\Spec}{\operatorname{Spec}}

\renewcommand{\O}{{\mathcal O}}
\renewcommand{\hom}{\operatorname{Hom}}
\newcommand{\real}{{\mathbb R}}
\newcommand{\cplx}{{\mathbb C}}

\newcommand{\Z}{{\mathbb Z}}
\newcommand{\R}{{\mathbb R}}
\newcommand{\PP}{{\mathbb P}}

\newcommand{\pone}{{\mathbb P}^1}

\newcommand{\im}{\operatorname{Im}}

\newcommand{\Ob}{\operatorname{Ob}}

\newcommand{\gm}{\mathbb{G}_{m}}

\newcommand{\Aut}{\operatorname{Aut}}
\newcommand{\Pic}{\operatorname{Pic}}

\newcommand{\bbP}{\mathbb{P}}

\newcommand{\dP}{\mathsf{dP}}

\newcommand{\leftexp}[2]{{\vphantom{#2}}^{#1}{#2}}
\newcommand{\weezer}{\leftexp{=}{\kern-0.23em\mathsf{W}}^{\kern-0.21em =}}

\begin{document}

\title[Derived categories of arithmetic toric varieties]{On derived categories of arithmetic toric varieties}

\author[Ballard]{Matthew R Ballard}
\address{Department of Mathematics, University of South Carolina, 
Columbia, SC 29208}
\email{ballard@math.sc.edu}
\urladdr{\url{http://people.math.sc.edu/ballard/}}

\author[Duncan]{Alexander Duncan}
\address{Department of Mathematics, University of South Carolina, 
Columbia, SC 29208}
\email{duncan@math.sc.edu}
\urladdr{\url{http://people.math.sc.edu/duncan/}}

\author[McFaddin]{Patrick K. McFaddin}
\address{Department of Mathematics, University of South Carolina, 
Columbia, SC 29208}
\email{pkmcfaddin@gmail.com}
\urladdr{\url{http://mcfaddin.github.io/}}

\begin{abstract}
We begin a systematic investigation of derived categories of smooth
projective toric varieties defined over an arbitrary base field.
We show that, in many cases, toric varieties admit full exceptional
collections.
Examples include all toric surfaces, all toric Fano 3-folds, some toric
Fano 4-folds, the generalized del Pezzo varieties of Voskresenski\u\i \
and Klyachko, and toric varieties associated to Weyl fans of type $A$.
Our main technical tool is a completely general Galois descent result
for exceptional collections of objects on (possibly non-toric) varieties
over non-closed fields.
\end{abstract}

\maketitle
\addtocounter{section}{0}


\section{Introduction}

Recently, several intriguing threads relating derived categories and
arithmetic geometry have emerged and motivated general structure questions 
for $k$-linear triangulated categories for arbitrary
fields $k$.  Such exploration has yielded many nice applications as well as 
further enticing problems, see as a sampling \cite{AKW, AAGZ, ADPZ, HT, Honigs, LiebMaulSnow}. 
Meanwhile over $\mathbb{C}$, structural results for derived
categories seem to have deep implications for the underlying birational geometry, e.g.
\cite{AT, ABB, BB, BMMSS, KuznetsovCubic4fold, Vial}. Taking these together,
derived categories become an important invariant for studying birational
geometry over a general field \cite{AB}. A further benefit of this
noncommutative approach is direct utility for solving problems in algebraic $K$-theory,
for example \cite{MerkPan}.

With such tantalizing ties, one would like a fertile testing ground for
questions. In this paper, we begin a systematic study of one such area:
derived categories of arithmetic toric varieties. This area has the
following nice features:

\begin{itemize}
 \item rationality issues are deep in general but tractable in examples,
 \item robust tools already exist to investigate derived categories over
the separable closure,
 \item and specific questions are often amenable to computational
experimentation.
\end{itemize}

One of the best tools for understanding a derived category is an
\emph{exceptional collection} consisting of \emph{exceptional objects}.
As originally conceived in \cite{Beilinson}, an exceptional object of a $k$-linear derived category is one whose endomorphism algebra
is isomorphic to the base field $k$.
When $k$ is not algebraically closed, this definition is too restrictive and
instead we use the existing notion: an object of $\mathsf{D^b}(X)$ is \emph{exceptional} if its endomorphism algebra is a division algebra. 
Details are discussed in Section~\ref{sect:galdesc} below.

We illustrate this more general notion for two arithmetic toric
varieties.
The real conic $X = \{x^2 + y^2 + z^2 = 0\} \subset \bbP^2_{\real}$
has an exceptional collection $\{ \O, \mathcal{F} \}$, where
$\text{End}(\mathcal{F})$ is isomorphic to the quaternion algebra
$\mathbb{H}$.
Over $\cplx$, we have $X_\cplx \simeq \pone_\cplx$ and $\mathcal{F}\otimes_{\real} \cplx \simeq \O(1)^{\oplus 2}$.
As another example, consider the Weil restriction $Y$ of $\pone_\cplx$ over
$\real$ (``$\pone(\cplx)$ viewed as an $\real$-variety'').
Here $Y$ has an exceptional collection $\{ \O, \mathcal{G},
\mathcal{H}\}$ where $\text{End}(\mathcal{G}) \simeq \cplx$ and
$\text{End}(\mathcal{H})
\simeq \real$.
Over $\cplx$, we have $Y\otimes_{\real} \cplx \simeq \pone \times \pone$
with
$\mathcal{G} \otimes_{\real} \cplx
\simeq \O(1,0) \oplus \O(0,1)$ and $\mathcal{H} \otimes_{\real} \cplx\simeq  \O(1,1)$,
where $\O(i, j) = \pi_1^*\O(i) \otimes \pi_2 ^*\O(j)$.

A central question for derived categories of arithmetic toric varieties is the following: 

\begin{question}\label{quest:main}
 Let $X$ be a smooth projective toric variety over an arbitrary field.  Does $X$ admit a full exceptional collection? If so, does $X$ possess a full exceptional collection of sheaves?
\end{question}
 
Over an algebraically-closed field of characteristic zero, there is always a full exceptional collection of objects \cite{Kawamata,Kawamata2} while the question of a full exceptional collection of sheaves is due to Orlov. 
Making allowances for different language,
the question is also known to have a positive answer for
Severi-Brauer varieties~\cite{AB,Bernardara},
minimal toric surfaces~\cite{BSS},
and smooth projective toric varieties with absolute Picard rank at most
$2$~\cite{Yan}.

In this article, we provide further evidence for a positive answer to
Question~\ref{quest:main},
treating cases with low dimension or a high degree of symmetry.

\begin{thm} \label{thm:examples}
 The following possess full exceptional collections of sheaves:
 \begin{itemize}
  \item smooth toric surfaces (Proposition~\ref{prop:surface}),
  \item smooth toric Fano 3-folds (Proposition~\ref{prop:3fold}), 
  \item all forms of 43 of the 124 smooth split toric Fano 4-folds (Section~\ref{sect:4fold}),
  \item all forms of centrally symmetric toric Fano varieties (Corollary~\ref{cor:centsym}), and
  \item and all forms in characteristic zero of toric varieties corresponding to Weyl fans of root systems of type $A$ (Proposition~\ref{prop:X(An)excpcoll}).
 \end{itemize}
\end{thm}

Our results leverage extant work in the algebraically closed case
such as \cite{Uehara} for 3-folds and\cite{Prabhu} for 4-folds. We use Castravet and Tevelev's recently discovered exceptional collection for $X(A_n)$ \cite{CT}. 
For the centrally symmetric toric Fano varieties (which are
products of ``generalized del Pezzo varieties'' and projective lines
\cite{VosKly}), we use an explicit exceptional
collection (see also \cite{BDMdP}) closely related to the one found in \cite{CT}.
Up to a twist by a line bundle, the authors had independently
discovered the exact same collection!
This suggests that symmetry imposes strong conditions on the
possible exceptional collections, which paradoxically makes them easier to find.

To study arithmetic exceptional collections, we establish
an effective Galois descent result for general exceptional collections.
This applies to general varieties, not just toric ones.

\begin{thm}[Theorem~\ref{thm:descblocks}, Lemma~\ref{lem:Galconverse}]
 Let $X$ be a $k$-scheme and $L/k$ a $G$-Galois extension.  Then $X_L$ admits a full (resp. strong) $G$-stable exceptional collection of objects of $\mathsf{D^b}(X_L)$ (resp. sheaves, resp. vector bundles) if and only if $X$ admits a full (resp. strong) exceptional collection of objects of $\mathsf{D^b}(X)$ (resp. sheaves, resp. vector bundles).
\end{thm}

We highlight one corollary of a positive answer to
Question~\ref{quest:main}. Arithmetic toric varieties are also studied
in \cite{MerkPan}, which focused on computing their algebraic $K$-groups
via decompositions in a certain noncommutative motivic category of
$K_0$-correspondences. They showed that for an arithmetic toric
$k$-variety $X$ with $G = \text{Gal}(k^s/k)$, the group $K_0(X_{k^s})$
is a direct summand of a \emph{permutation $G$-module} (there exists a
$\Z$-basis permuted by $G$).

\begin{question}[Merkurjev-Panin \cite{MerkPan}]\label{quest:MP}
Let $X$ be an arithmetic toric variety over $k$ and $G = \operatorname{Gal}(k^s/k)$.  Is $K_0(X_{k^s})$ always a permutation $G$-module?
\end{question}

Question~\ref{quest:main} can be viewed as a categorification of Question~\ref{quest:MP} as any such exceptional collection over $k$ immediately gives a permutation basis. 

In order to show that every toric variety has a full exceptional collection
over $\cplx$, the main tool used in \cite{Kawamata,Kawamata2} was the
minimal model program (MMP) in birational geometry.
The basic building blocks are toric stacks with Picard rank one,
which always have full strong exceptional collections of line bundles.
Indeed, runs of the MMP can be leveraged to effectively produce
exceptional collections \cite{BFK}.

Over a non-closed field, one hopes to use the Galois-equivariant MMP,
but the situation is more complicated.
The most basic building blocks in this framework are those varieties $X$
which have $\rho^G = \text{rank} (\text{Pic}(X)^G )= 1$.
Based on the results above and the hope of using the MMP in the
arithmetic situation, we ask the following
question in the vein of \cite{King,BH,CM-RFrob}:

\begin{question}\label{quest:invpicrank1}
Let $X$ be a smooth toric $k$-variety and $L/k$ a $G$-Galois splitting
field.  If $\operatorname{Pic}(X_L)^G$ is of rank 1, does $X_L$ possess
a full strong $G$-stable exceptional collection consisting of line
bundles?
\end{question}

\subsection*{Acknowledgements}

The first author was partially supported by NSF DMS-1501813. He would
also like to thank the Institute for Advanced Study for providing a
wonderful research environment. These ideas were developed during his
membership.  The first author benefited from discussions with Alicia
Lamarche.
The second author was partially supported by NSA grant
H98230-16-1-0309.
The third author would like to thank the Hausdorff Institute for their
hospitality and lively research environment during the \emph{K-theory
and related fields} trimester program. A large portion of this
manuscript was drafted during his time in Bonn.
All authors also thank Fei Xie for pointing out that, due to an editing
error, in a previous version of this paper,
Proposition~\ref{prop:surface} stated that all smooth toric surfaces
have strong collections of vector bundles instead of a not-necessarily
strong collection of sheaves as was proven.
The authors would also like to thank an anonymous referee for useful
comments.


\subsection*{Organization}

Section~\ref{sect:galdesc} treats Galois descent of exceptional
collections consisting of objects on (possibly non-toric) varieties. In
Section~\ref{section:toric}, we recall appropriate definitions of
arithmetic toric varieties and establish additional descent results
which are specific to toric varieties.  In
Section~\ref{section:minimal}, we consider a range of examples.  We begin by treating toric surfaces, followed by toric Fano 3-folds.  For toric Fano 4-folds, we give partial results. We conclude by investigating the class of centrally symmetric toric Fano varieties, including the generalized del Pezzo varieties, and handling toric varieties associated to root systems of type $A$. 


\subsection*{Notation} Throughout, $k$ denotes an arbitrary field and
$k^s$ a separable closure. A \emph{variety} is a geometrically integral
separated scheme of finite type over $k$. All our schemes will be
quasi-compact and quasi-separated. For a $k$-scheme $X$ and field
extension $L/k$, we write $X_L : = X \times _{\Spec k} \Spec L$.  If $A$
is a $k$-algebra, we write $A_L = A\otimes _{k} L$.
We use $\mathsf{D^b}(X)$ to denote the bounded derived category
$\mathsf{D^b}(\text{Coh}(X))$. For an $\O_X$-algebra $A$, we write
$\mathsf{D^b}(A)$ for the bounded derived category of complexes of
$A$-modules which are coherent $\O_X$-modules.


\section{Galois descent and exceptional collections}\label{sect:galdesc}

In this section we develop Galois descent for exceptional collections (in a generalized sense). We begin by recalling some definitions and conventions concerning structure theory of derived categories of schemes.  We then give our main descent results for $G$-stable exceptional collections (Theorem~\ref{thm:descblocks}). We complete the section by collecting a few useful consequences to be used in the sequel.


\subsection{Exceptional collections}

We give some conventions for semiorthogonal decompositions of derived categories and in particular exceptional collections.  Such collections have been widely studied over algebraically closed fields but have recently been treated in more generality \cite{AAGZ, AB, ABB, Bernardara, BSS, Elagin, Xie, Yan}.  We refer the reader to Remarks~\ref{rem:descSOD} and \ref{rem:elagin} for added details on some of these results.  

For a triangulated category $\mathsf{T}$, we use the standard notation $\text{Ext}^n_{\mathsf{T}}(A, B) = \hom _{\mathsf{T}} (A, B[n])$.  For objects $A, B $ of $\mathsf{D^b}(X)$, we use  $\text{End}_X(A)$ and $\text{Ext}_X^n(A, B)$ to denote $\text{End}_{\mathsf{D^b}(X)}(A)$ and $\text{Ext}^n_{\mathsf{D^b}(X)}(A, B)$, respectively.

\begin{defn}[see \cite{BK}]
Let $\mathsf{T}$ be a triangulated category.  A full triangulated subcategory of $\mathsf{T}$ is \emph{admissible} if its inclusion functor admits left and right adjoints.  A \emph{semiorthogonal decomposition} of $\mathsf{T}$ is a sequence of admissible subcategories $\mathsf{C}_1, ..., \mathsf{C}_s$ such that 
\begin{enumerate}
\item $\hom _{\mathsf{T}}(A_i, A_j) = 0$ for all $A_i \in \Ob (\mathsf{C}_i)$, $A_j \in \Ob (\mathsf{C}_j)$ whenever $i > j$.
\item For each object $T$ of $\mathsf{T},$ there is a sequence of morphisms $0 = T_s \to \cdots \to T_0 = T$ such that the cone of $T_i \to T_{i-1}$ is an object of $\mathsf{C}_i$ for all $i = 1,..., s$.
\end{enumerate}
We use  $\mathsf{T} = \langle \mathsf{C}_1,..., \mathsf{C}_s\rangle$ to denote such a decomposition.
\end{defn}

Particularly nice examples of semiorthogonal decompositions are given by exceptional collections, the study of which goes back to Beilinson \cite{Beilinson}.

\begin{defn}\label{def:exceptional}
Let $\mathsf{T}$ be a $k$-linear triangulated category.  An object $E$ in $\mathsf{T}$ is \emph{exceptional} if the following conditions hold:
\begin{enumerate}
\item $\text{End}_{\mathsf{T}}(E)$ is a division $k$-algebra.
\item $\text{Ext}^n_{\mathsf{T}}(E, E) = 0$ for $n \neq 0$.
\end{enumerate}
A totally ordered set $\mathsf{E} = \{E_1, ..., E_s\}$ of exceptional objects is an \emph{exceptional collection} if $\text{Ext}^n_{\mathsf{T}}(E_i, E_j) = 0$ for all integers $n$ whenever $i >j$.  An exceptional collection is $\emph{full}$ if it generates $\mathsf{T}$, i.e., the smallest thick subcategory of $\mathsf{T}$ containing $\mathsf{E}$ is all of $\mathsf{D^b}(X)$.  An exceptional collection is \emph{strong} if $\text{Ext}^n_{\mathsf{T}}(E_i, E_j) = 0$ whenever $n \neq 0$.  An \emph{exceptional block} is an exceptional collection $\mathsf{E} = \{ E_1, ..., E_s\}$ such that $\text{Ext}^n_{\mathsf{T}}(E_i, E_j) = 0$ for every $n$ whenever $i \neq j$. Given an exceptional collection $\mathsf{E} = \{E_1, ..., E_s\}$, we denote by $\langle \mathsf{E} \rangle$ the category generated by the objects $E_i$.
\end{defn}

\begin{rem}
Our notion of exceptional object generalizes the classical one, where item $(1)$ of Definition~\ref{def:exceptional} is replaced by:  $\text{End}_{\mathsf{T}}(E) = k$ \cite[$\S$2]{Bondal}.  Over algebraically or separably closed fields, these definitions agree.  Over non-closed fields, the classical definition is too restrictive to allow for the use of interesting arithmetic invariants in the study of exceptional collections on twisted forms, e.g., Brauer classes.
\end{rem}

\begin{prop}[Thm. 3.2 \cite{Bondal}]\label{prop:exctosod}
Let $X$ be a $k$-scheme with exceptional collection $\{E_1,..., E_s\}$.  If $\mathscr{E}_i $ is the category generated by $E_i$, there is a semiorthogonal decomposition $\mathsf{D^b}(X) = \langle \mathscr{E}_1,..., \mathscr{E}_s, \mathsf{A} \rangle$, where $\mathsf{A}$ is the full subcategory with objects $A$ such that $\operatorname{Hom}_X(A, E_i) = 0$ for all $i$.
\end{prop}

\begin{rem}
 The reference assumes smoothness and projectivity but the conclusion is independent of this. Note further that if $X$ admits a full exceptional collection then it is automatically smooth and proper by \cite[Propositions 3.30 and 3.31]{OrlovNCstuff}.
\end{rem}

The existence of an exceptional collection on a scheme $X$ provides a means of studying derived geometry of $X$ in purely algebraic terms.  Indeed, in such a situation, one may identify an ``underlying" $k$-algebra which is derived equivalent to $X$.  For exceptional blocks, one obtains a similar but slightly stronger fact.

\begin{prop}[Thm. 6.2 \cite{Bondal}]\label{thm:tilting}
Let $X$ be a smooth projective $k$-scheme and let $\{E_1,..., E_n\}$ be a full strong exceptional collection on $\mathsf{D^b}(X)$.  Let $\mathcal{E}  = \bigoplus E_i$ and $A = \operatorname{End}(\mathcal{E})$.  Then $\mathsf{R} \hom _{\mathsf{D^b}(X)}(\mathcal{E}, -) : \mathsf{D^b}(X) \to \mathsf{D^b}(A)$ is a $k$-linear equivalence.
\end{prop}

\begin{prop}
If $\mathsf{E} = \{E_1,..., E_s\}$ is an exceptional block with $\operatorname{End}(E_i) = D_i$, there is a $k$-algebra isomorphism $\operatorname{End}(\bigoplus E_i) \simeq D_1 \times \cdots \times D_s$, and hence a $k$-linear equivalence $\langle \mathsf{E} \rangle \simeq  \mathsf{D^b}(D_1 \times \cdots \times D_n)$.
\end{prop}

The object $\mathcal{E} = \oplus E_i$ of Proposition~\ref{thm:tilting} is usually called a \emph{tilting object}.  If each $E_i$ is a sheaf (resp. vector bundle), then $E$ is called a \emph{tilting sheaf} (resp. \emph{tilting bundle}).  Until recently, the theory of tilting objects has served as the main tool for extending the study of exceptional collections to non-closed fields.  The results above show that any exceptional collection gives rise to both a tilting object and a semiorthogonal decomposition, and thus the admission of such a collection is a particularly special property of a given triangulated category.  Our aim in the following subsection is to extend descent results for semiorthogonal decompositions and tilting objects to (our more general notion of) exceptional collections. We give a formal definition of tilting object for completeness.

\begin{defn}
A \emph{tilting object} for a $k$-scheme $X$ is an object $\mathcal{E}$ of $\mathsf{D^b}(X)$ which satisfies the following conditions:
\begin{enumerate}
\item $\text{Ext}_X ^n (\mathcal{E}, \mathcal{E}) = 0$ for $n > 0$.

\item $\mathcal{E}$ generates $\mathsf{D^b}(X)$.
\end{enumerate}
\end{defn}

\begin{rem}[$K$-theory and motivic decompositions]

Exceptional collections have a particularly interesting manifestation in the realm of noncommutative motives.  Indeed, an exceptional collection $\{E_1,..., E_s\}$ on a smooth projective variety $X$ yields a decomposition $U(X) \simeq \bigoplus _i U(D_i)$ of its corresponding universal additive invariant \cite[$\S$2.3]{Tabuada}, where $D_i = \text{End}(E_i)$. This defines a motivic decomposition by viewing $X$ as an object in the Merkurjev-Panin category of $K$-motives \cite{MerkPan} or Kontsevich's category of noncommutative Chow motives \cite[Thm. 6.10]{Tabuada2} via its associated dg-category of perfect complexes.

One nice consequence is that this decomposition is detected by algebraic $K$-groups \cite[Prop. 1.10]{AB} in addition to a slew of other additive invariants in the sense of Tabuada \cite[$\S$2.2]{Tabuada}. Such invariants include algebraic $K$-theory with coefficients, homotopy $K$-theory, \'{e}tale $K$-theory, (topological) Hochschild homology, and (topological) cyclic homology.
\end{rem}


\subsection{Galois descent} We develop Galois descent for exceptional
collections consisting of objects in the derived category
$\mathsf{D^b}(X)$ of a (smooth projective) variety $X$.  Throughout this
section, pushforward and pullback functors are understood to be derived.
For a $k$-scheme $X$ and finite Galois extension $L/k$, any element $g
\in \text{Gal}(L/k)$ defines a morphism of $k$-schemes $g: X_L \to X_L$
which in turn defines the functor $g^*: \mathsf{D^b}(X_L) \to \mathsf{D^b}(X_L)$.

\begin{defn}
Let $X$ be a scheme with an action of a group $G$.  A $G$-\emph{stable
exceptional collection} on $X$ is an exceptional collection $\mathsf{E}
= \{E_1, ..., E_s\}$ of objects in $\mathsf{D^b}(X)$ such that
for all $g \in G$ and $1 \leq i \leq s$ there exists $E \in \mathsf{E}$
such that $g^*E_i \simeq E$.
We say a $G$-stable exceptional collection $\mathsf{E}$
is a \emph{$G$-orbit} if, for every pair of objects
$E,E' \in \mathsf{E}$, there exists a $g \in G$ such that
$g^*E \simeq E'$.
\end{defn}

\begin{rem}\label{rem:invariant}
A simple example of a $G$-stable exceptional collection is a
$G$-\emph{invariant} exceptional collection, i.e., an exceptional
collection $\{E_1,..., E_s\}$ such that $g^*E_i \simeq E_i$ for all $1 \leq i \leq s$.  It is often the case that toric varieties admit exceptional collections consisting of line bundles.  If it is also the case that a group $G$ acts trivially on $\text{Pic}(X)$, such a collection is automatically $G$-invariant, and hence $G$-stable (see Lemma~\ref{lem:picinv}).
\end{rem}

\begin{lem}\label{lem:collectiontoblocks}
Any $G$-stable exceptional collection may be written as a collection of $G$-stable exceptional blocks (after possibly reordering).
\end{lem}

\begin{proof}
The decomposition of a $G$-stable exceptional collection into its $G$-orbits gives the desired exceptional blocks.  Let $\mathsf{E}$ be a $G$-stable exceptional collection and for elements $E, E' \in \mathsf{E}$, we write $E \leadsto E'$ if $\text{Ext}^n(E, E') \neq 0$ for some $n$.

 Let $\mathsf{A} \subset \mathsf{E}$ be a $G$-orbit. To see that
$\mathsf{A}$ is an exceptional block, suppose that $E  \leadsto E'$ for
$E, E' \in \mathsf{A}$. Since $\mathsf{A}$ is a $G$-orbit, $E' \simeq g^*
E$ for some $g \in G$.  Thus, $E  \leadsto g^* E$, and acting again by
$g$, we have $g^*E  \leadsto (g^2)^*E$. Since $A$ is finite, we have $E
\leadsto g^*E  \leadsto \cdots  \leadsto (g^s)^*E  \leadsto E$ for some
positive integer $s$.  Thus,
there is no ordering of the elements of $\mathsf{A}$
such that they form a subset of an exceptional collection --- a
contradiction.

If $\mathsf{B}$ is another $G$-orbit (distinct from $\mathsf{A}$), we
would like to see that these blocks can be ordered to form an
exceptional collection. We claim that for any $E \in \mathsf{A}$ and $F
\in \mathsf{B}$, one has $E  \leadsto F$ only if 
$\mathsf{A}$ precedes
$\mathsf{B}$ in the collection $\mathsf{E}$ (i.e., $\text{Ext}^n(B, A) =
0$ for all $n$ and all $A \in \mathsf{A}$, $B \in \mathsf{B}$).
To see this, assume that $E  \leadsto
F$ and $F  \leadsto E'$ for some $E' \in \mathsf{A}.$ As $\mathsf{A}$ is
a $G$-orbit, $E' \simeq g^*E$ for some $g \in G$. Hence, just as above,
we have a sequence $E  \leadsto F  \leadsto g^*E \leadsto g^* F
\leadsto \cdots  \leadsto (g^s)^* F  \leadsto E.$  Thus, there is no
ordering of the elements of $\mathsf{A}$ and $\mathsf{B}$ which forms an
exceptional collection, contradicting the exceptionality of
$\mathsf{E}$.
\end{proof}

\begin{lem}\label{lem:galconj}
Let $X$ be a Noetherian $k$-scheme, $L/k$ a finite Galois extension with group $G$, and $\pi: X_L \to X$ the natural projection map.  For any object $M $ in $\mathsf{D^b}(X_L)$ there is a natural isomorphism $\displaystyle \pi^* \pi_*( M) \simeq \bigoplus _{g\in G} g^*M.$  
\end{lem}

\begin{proof}

As $\pi$ is flat and affine, every coherent sheaf on $X$ is acyclic for $\pi^\ast$ and every coherent sheaf on $X_L$ is acyclic for $\pi_\ast$. Hence, the derived functors coincide with the application of $\pi^\ast$ or $\pi_\ast$ component-wise to a complex. Thus, it suffices to establish a natural isomorphism at the level of coherent sheaves. 

For any object $M$ of $\text{Coh}(X_L)$, we have $ \pi_*M \simeq \pi_*
g^* M $, and adjunction yields a natural transformation $ \pi^* \pi_*
\to g^*$.  Summing over all $g \in G$ provides the transformation
$\alpha:   \pi^* \pi_* \to \oplus g^* $. We show this is an
isomorphism.

It suffices to check that $\alpha$ is an isomorphism on any affine patch, $\operatorname{Spec} R$, of $X$. Passing to modules, we abuse notation and let $M$ be a finitely-generated module over $R_L = R \otimes_k L$. Choose a presentation of $M$
\begin{displaymath}
 R_L^{\oplus m} \to R_L^{\oplus n} \to M \to 0
\end{displaymath}
and evaluate $\alpha$ on the sequence to get the commutative diagram 
\begin{center}
 \begin{tikzpicture}
  \node at (-4,0.75) (pr1) {$R^{\oplus m} \otimes_k \left( L \otimes_k L \right)$};
  \node at (-4,-0.75) (gr1) {$R^{\oplus m} \otimes_k \left( \oplus_g \Gamma_g(L) \right)$};
  \node at (0,0.75) (pr0) {$R^{\oplus n} \otimes_k \left( L \otimes_k L\right) $};
  \node at (0,-0.75) (gr0) {$R^{\oplus m} \otimes_k \left( \oplus_g \Gamma_g(L)\right) $};
  \node at (3,0.75) (pm) {$M \otimes_R R_L$};
  \node at (3,-0.75) (gm) {$\oplus_g g^\ast M$};
  \node at (5,0.75) (p0) {$0$};
  \node at (5,-0.75) (g0) {$0$};
  \draw[->] (pr1) -- (pr0);
  \draw[->] (pr0) -- (pm);
  \draw[->] (pm) -- (p0);
  \draw[->] (gm) -- (g0);
  \draw[->] (gr1) -- (gr0);
  \draw[->] (gr0) -- (gm);
  \draw[->] (pr1) -- node[left] {$\alpha_{R^{\oplus m}}$} (gr1);
  \draw[->] (pr0) -- node[left] {$\alpha_{R^{\oplus n}}$} (gr0);
  \draw[->] (pm) -- node[left] {$\alpha_M$} (gm);
 \end{tikzpicture}
\end{center}
where $\Gamma_g(L)$ denotes the graph of $g$ in $L \otimes_k L$. The left and middle maps are isomorphisms, so the right map must also be an isomorphism. 
\end{proof}

\begin{prop}[Descent for orbits]\label{prop:objblockdescent}
Let $X$ be a $k$-scheme, $L/k$ a finite $G$-Galois extension,
and $\pi: X_L \to X$ the natural projection map. If $\mathsf{E} = \{E_1,
\ldots, E_s\}$ is a $G$-orbit forming an exceptional collection on $X_L$,
and if $E$ is any element of $\mathsf{E}$, then there is an exceptional
object $F$ in $\mathsf{D^b}(X)$ such that $\pi_*E \simeq F^{\oplus m}$
and $\pi^*F$ generates the category $ \langle \mathsf{E} \rangle$.
\end{prop}

\begin{proof}
By Lemma~\ref{lem:collectiontoblocks}, exceptional $G$-orbits are
completely orthogonal (and by definition carry a transitive action of
$G$), which will be used throughout the proof. Fix an element $E \in
\mathsf{E}$, so that $E = E_i$ for some $i$. Lemma~\ref{lem:galconj}
gives $$\pi^*\pi_*E \simeq \bigoplus _{g \in G} g^*E$$  We claim that
$\text{End}(\pi_*E)$ is a matrix algebra over a division algebra, and
prove this by first showing that it is semisimple. Indeed, using
$\text{End}_X(M) \otimes _k L \simeq \text{End}_{X_L} (\pi^*M)$ for any $M
\in \mathsf{D^b}(X)$ \cite[Rem. 2.1]{AB}, we have
\[ \text{End}_X(\pi_*E) \otimes _k L \simeq
\text{End}_{X_L}(\pi^* \pi_* E) \simeq
\text{End}_{X_L}\left(\bigoplus _{ g\in G} g^*E\right).
\]  Each $g^*E$ is exceptional so that
$\text{End}_{X_L}(g^*E) =: D_g$ is a division algebra for each element $g
\in G$.  Let  $H \leq G$ be the subgroup consisting of elements $h$
satisfying $h^*E \simeq E$. For any system of coset representatives $g
\in G/H$, we have $\text{End}_X(\pi_*E)_L \simeq \prod _{g \in
G/H}M_m(D_g)$, where $m = |H|$. This product of matrix algebras over
division algebras is semisimple, i.e., the Jacobson radical
$\text{rad}(\text{End}_X(\pi_*E)_L) = 0$.  We then have $0 =
\text{rad}(\text{End}_X(\pi_*E)_L) = \text{rad}(\text{End}_X(\pi_*E))_L$
by ~\cite[Thm.~1]{Amitsur}, and hence $\text{rad}(\text{End}_X(\pi_*E))
= 0$. Thus, $\text{End}_X(\pi_*E)$ is semisimple and so must also
be a product of matrix algebras over division algebras by the
Artin-Wedderburn Theorem.  

Let $Z$ be the center of $\text{End}_X(\pi_*E)$ and $Z_L$ the center of $\text{End}_X(\pi_*E)_L$.
Note that $Z$ is an \'{e}tale $k$-algebra, and to show that $\text{End}(\pi_*E)$ is a matrix algebra, it suffices to show that $Z$ has no zero divisors, and is thus a field.  There is an embedding $Z \hookrightarrow Z_L  = \prod_{g \in G/H}  L_g$, where $L_g$ is the center of the division algebra $D_g$.  The transitive action of $G$ on $\{E_1,..., E_s\}$ implies that $G$ acts transitively on a basis of $Z_L$, so that $Z = (Z_L)^G$ has no zero divisors.

We produce the object $F$ using the identification $\text{End}_X(\pi_*E)
\simeq M_n(D)$, where $D$ is a division algebra.
Let $e_i = e_{ii}$ denote the usual idempotent matrices, so that
$\{e_i\}$ is a complete set of primitive orthogonal idempotents.
Notice that $F_i:= \im (e_i)$ is a simple $\text{End}_X(\pi_*E)$-submodule of $\pi_*E$ for each $i$, and hence $F_i \simeq F_j$ for each $i, j$, and $\text{End}_X(F_i) \simeq D$. Define $F = \im (e_1) \subset \pi_*E$, included as a direct summand.  We note that $\pi_*E \simeq \bigoplus F_i \simeq F^{\oplus n}$.  

We now show that $F$ is an exceptional object on $X$.  As stated above, $\text{End}_X(F)$ is a division algebra, so it suffices to show that $\text{Ext}^n_X(F, F) = 0$ for $ n \neq 0$.  Using Lemma~\ref{lem:galconj} and $(\pi^*, \pi_*)$-adjunction, we have $$\text{Ext}^n_X(\pi_*E, \pi_*E) = \bigoplus _{g\in G} \text{Ext}^n _{X_L}(g^*E, E).$$  For $n \neq 0$, each summand of the right-hand side is 0, which follows from the mutual orthogonality of the exceptional block $\mathsf{E}$ (when $g^*E \not\simeq E$) and from exceptionality of $E$ (when $g^*E \simeq E$).  Since $F$ is a direct summand of $\pi_*E$, it follows that $\text{Ext}^n_X(F, F)$ is a summand of $\text{Ext}^n_X(\pi_*E, \pi_*E) = 0$.

Lastly, we show that $\pi^*F$ generates the category $\langle \mathsf{E}
\rangle$.   Since $ F^{\oplus m} \simeq \pi_*E$, extending scalars to $L$
gives $ (\pi^*F)^{\oplus m} = \pi^*(F^{\oplus m}) \simeq \pi^*\pi_*E
\simeq \bigoplus g^*E$.  Thus, $$\langle \pi^*F \rangle = \langle (\pi ^*F )^{\oplus m} \rangle = \langle \bigoplus g^* E \rangle = \langle g^*E \rangle _{g \in G} = \langle \mathsf{E} \rangle.$$\end{proof}

\begin{rem}\label{rem:descSOD}
Proposition~\ref{prop:objblockdescent} provides a very specific case of descent
for triangulated categories, the main advantage of which is that it
allows one to identify a specific exceptional object that base extends
to the given orbit.
Moreover, a $G$-orbit which forms an exceptional collection consisting of vector bundles (resp. sheaves) descends to an exceptional collection consisting of vector bundles (resp. sheaves).  Compare to the following descent result for semiorthogonal decompositions, which generalizes \cite[Cor. 2.15]{Toen}.  Although this result is useful for descending semiorthogonal decompositions, it does not identify exceptional objects.

\begin{prop}[Prop. 2.12, \cite{AB}]\label{prop:ABdesc}
Let $\mathsf{T}$ be a $k$-linear triangulated category such that $\mathsf{T}_{k^s}$ is $k^s$-equivalent to $\mathsf{D^b}(k^s, (k^s)^n)$.  Then there exists an \'{e}tale algebra $K$ of degree $n$ over $k$, an Azumaya algebra $A$ over $K$, and a $k$-linear equivalence $\mathsf{T} \simeq \mathsf{D^b}(K/k, A)$.
\end{prop}

\noindent Let $X$, $\mathsf{E}$, and $F$ be as in Proposition~\ref{prop:objblockdescent}, and note that taking $\mathsf{T} = \langle F \rangle$, we have $\mathsf{T}_{k^s} = \langle \pi^* F \rangle_{k^s} = \langle \mathsf{E} \rangle _{k^s}$. Since $\mathsf{E} = \{ g^*E\}_{g \in G}$ is a full exceptional collection for $\langle \mathsf{E} \rangle$, the bundle $\mathcal{E} : = \oplus (g^*E)_{k^s}$ is a tilting object for $\langle \mathsf{E}\rangle_{k^s}$.  This defines an equivalence $$ \mathsf{T}_{k^s} \simeq \langle \mathsf{E} \rangle_{k^s} \simeq \mathsf{D^b}(k^s, \text{End}(\mathcal{E})) = \mathsf{D^b}(k^s, (k^s)^n).$$ Proposition~\ref{prop:ABdesc} yields an \'{e}tale extension $K/k$, an Azumaya $K$-algebra $A$, and an equivalence $\mathsf{T} \simeq \mathsf{D^b}(K/k, A)$.  In this case, since $\mathsf{T} = \langle F \rangle$, we see that $A = \text{End}_X(F)$ is an Azumaya algebra over its center $Z$ (using the notation found in the proof of Proposition~\ref{prop:objblockdescent}), which is simply a field extension of $k$.
\end{rem}

\begin{thm}[Descent for stable collections]\label{thm:descblocks}
Let $X$ be a $k$-scheme, $L/k$ a finite $G$-Galois extension, and $\pi: X_L \to X$ the natural projection map.  If $X_L$ admits a full $G$-stable exceptional collection $\mathsf{E}$ of objects of $\mathsf{D^b}(X_L)$, then $X$ admits a full exceptional collection $\mathsf{F}$ of objects of $\mathsf{D^b}(X)$.  If $\mathsf{E}$ is strong, so is $\mathsf{F}$.  If the elements of $\mathsf{E}$ are vector bundles (resp. sheaves), the elements of $\mathsf{F}$ are vector bundles (resp. sheaves).
\end{thm}

\begin{proof}

By Lemma~\ref{lem:collectiontoblocks}, we may write $\mathsf{E} =
\{\mathsf{E}^1,..., \mathsf{E}^s\}$ as a collection of $G$-stable
blocks, where each block is given by a $G$-orbit.
Proposition~\ref{prop:objblockdescent} then associates to each block
$\mathsf{E}^i$ an exceptional object $F_i$ on $X$, and we show that
$\mathsf{F} = \{F_1,..., F_s\}$ is a full exceptional collection on $X$.
We first show that $\text{Ext}^n_{X}(F_i, F_j) = 0$ for all $n$ whenever
$i > j$.  Let $E^i$ and $E^j$ be elements of the collections
$\mathsf{E}^i$ and $\mathsf{E}^j$, respectively.  We then have
\begin{equation}
\text{Ext}^n_X(\pi_*E^i, \pi_*E^j) \simeq
\bigoplus _{g\in G} \text{Ext}^n _{X_L}(g^*E^i, E^j).
\label{eq:1}
\end{equation}
Since
$E^i$ and $E^j$ are elements of the exceptional collection $\mathsf{E}$
and $i < j$, each summand is 0 for all $n$, so that
$\text{Ext}_X^n(\pi_*E^i, \pi_*E^j) = 0$ for all $n$.  The objects $F_i$
and $F_j$ are direct summands of $\pi_*E^i$ and $\pi_*E^j$,
respectively, and it follows that $\text{Ext}^n_X(F_i, F_j) = 0$ for all
$n$.

By Proposition~\ref{prop:exctosod}, the exceptional collection $\{ F_1,
\ldots, F_s\}$ yields a semiorthogonal decomposition
\[
\mathsf{D^b}(X) = \langle \mathscr{F}_1, \ldots, \mathscr{F}_s, \mathsf{A}\rangle,
\]
where
$\mathscr{F}_i = \langle F_i \rangle$ and $\mathsf{A}$ is the full
subcategory of objects $A$ with $\text{Hom}_{\mathsf{D^b}(X)}(A, F_i) =
0$ for all $i$.  In particular, the subcategories $\mathscr{F}_i$ are
admissible.  Extending scalars to $L$, we have $(\mathscr{F}_i)_L =
\langle \mathsf{E}^i\rangle$, as both categories are generated by
$\pi^*F$ by Proposition~\ref{prop:objblockdescent}. The exceptional collection
$\mathsf{E} =\{\mathsf{E}^1,\ldots, \mathsf{E}^s \}$ is full, hence we have
a semiorthogonal decomposition $$ \mathsf{D^b}(X_L) =  \langle
(\mathscr{F}_1)_L, \ldots, (\mathscr{F}_s)_L \rangle.$$  Since our
admissible subcategories $\mathscr{F}_i$ base extend to a semiorthogonal
decomposition, \cite[Lem. 2.9]{ABB} gives a semiorthogonal decomposition
$\mathsf{D^b}(X) = \langle \mathscr{F}_1, \ldots, \mathscr{F}_s \rangle$.
In particular, the collection $\{F_1, \ldots, F_s\}$ generates
$\mathsf{D^b}(X)$, so this collection is full.

If $\mathsf{E}$ is strong, the right side of \eqref{eq:1} vanishes for $i \neq j$ (and any $n$).  It follows exactly as above that $\text{Ext}^n_{X}(F_i, F_j) = 0$ for all $n$ when $i \neq j$, so that $\mathsf{F}$ is strong.
\end{proof}

\begin{rem}\label{rem:elagin}
Similar descent results for collections of sheaves are given by Elagin in the algebraically closed case (i.e., $k = \bar{k}$) using the framework of equivariant exceptional collections in equivariant derived categories \cite{Elagin}.  Indeed, for a variety $X$ with an action of a finite group $G$ and a $G$-invariant exceptional collection (see Remark~\ref{rem:invariant}) consisting of sheaves, this descent result is given in terms of $\alpha$-twisted representations of $G$ (see Theorem 2.2 of loc. cit.).  For a $G$-stable exceptional collection consisting of sheaves, results are in terms of coinduced twisted representations of $G$ (see Theorem 2.3 of loc. cit.).
\end{rem}

\begin{lem} \label{lem:Galconverse}
Let $X$ be a $k$-scheme and $L/k$ a finite $G$-Galois extension.  If $X$ admits an exceptional collection, then $X_L$ admits a $G$-stable exceptional collection.
\end{lem}

\begin{proof}
 Let $E_1,\ldots,E_s$ be the exceptional collection on $X$ and consider $\pi^\ast E_1, \ldots, \pi^\ast E_s$ on $X_L$. To compute morphisms, we note that 
 \begin{displaymath}
  \operatorname{Hom}_{X_L} (\pi^\ast E_i , \pi^\ast E_j ) = \operatorname{Hom}_X( E_i, \pi_\ast \pi^\ast E_j) = \operatorname{Hom}_X(E_i, E_j \otimes_k L) = \operatorname{Hom}_X(E_i,E_j) \otimes_k L.
 \end{displaymath}
 This vanishes if $j > i$. Let $A_i = \operatorname{Hom}_X(E_i, E_i)$. We can split $A_i \otimes_k L$ as a product of matrix algebras over division algebras $A_{i,j} = M_{N_{i,j}}(D_{i,j})$ and correspondingly decompose
 \begin{displaymath}
  \pi^\ast E_i = \bigoplus F_{i,j}^{N_{i,j}}
 \end{displaymath}
 with 
 \begin{displaymath}
  \operatorname{Hom}_{X_L}( F_{i,j} , F_{i,j} ) = D_{i,j}.
 \end{displaymath}
 Note that $F_{i,j}$ and $F_{i,j^\prime}$ are orthogonal for $j \not = j^\prime$. Thus, we have an exceptional collection.
\end{proof}

\begin{lem}\label{lem:picinv}
Let $X$ be a $k$-scheme and $L/k$ a finite extension with Galois group $G$.  If $G$ acts trivially on $\operatorname{Pic}(X_L)$ and $X_L$ admits an exceptional collection of line bundles, then $X$ admits an exceptional collection of vector bundles.
\end{lem}

\begin{proof}
 The collection on $X_L$ is automatically $G$-stable pointwise. Hence we can apply Theorem~\ref{thm:descblocks}. 
\end{proof}

\begin{rem}
 Note that while we may start with a collection of line bundles,
the descended collection may not consist only of line bundles.
An example of this is the real conic discussed in the introduction.
\end{rem}

\begin{lem}\label{lem:blowup}
Let $X$ be a smooth $k$-variety and $L/k$ a $G$-Galois extension.  Let $Y_1, ..., Y_s$ be a $G$-orbit of smooth transversal subvarieties of $X_L$. Let $Y_I = \cap_{i \in I} Y_i$ and let $H_I$ be the normalizer of $Y_I$. If each $Y_I$ admits a full $H_I$-stable exceptional collection, then $\tilde{X}$ admits an exceptional collection, where $\tilde{X}_L$ is the iterated blow up of $X_L$ at the $Y_i$ (in any order).
\end{lem}

\begin{proof}
 This is an iterated application of Orlov's Theorem, see \cite[Lemma 7.2]{CT}.
\end{proof}


\section{Arithmetic toric varieties}\label{section:toric}

We introduce toric varieties over arbitrary fields. Such varieties, also known as \emph{arithmetic toric varieties}, have been treated in \cite{Duncan, ELFST, MerkPan, VosKly}.  
\begin{defn}

A \emph{torus} (over $k$) is an algebraic group $T$ (over $k$) such that $T _{k^s} \simeq \gm ^n$. A torus is \emph{split} if $T \simeq \gm ^n$.  A field extension $L/k$ satisfying $T_L \simeq \gm ^n$ is called a \emph{splitting field} of the torus $T$.  Any torus admits a finite Galois splitting field. 
\end{defn}

\begin{defn}
Given a torus $T$, a \emph{toric} $T$-\emph{variety} is a normal variety with a faithful $T$-action and a dense open $T$-orbit.  A toric $T$-variety is \emph{split} if $T$ is a split torus. A \emph{splitting field} of a toric $T$-variety is a splitting field of $T$. A variety is a \emph{toric variety} if it is a toric $T$-variety for some torus $T$.
\end{defn}

\begin{defn}
Let $X$ be a toric $T$-variety whose dense open $T$-orbit contains a $k$-rational point.  Then we say $X$ is \emph{neutral} \cite{Duncan} (or a \emph{toric} $T$-\emph{model} \cite{MerkPan}).
An orbit of a split torus always has a $k$-point, so a split toric
variety is neutral; but the converse is not true in general.
\end{defn}

\begin{rem}
In what follows, we will use the term \emph{toric variety} to mean toric $T$-variety for some fixed torus $T$, even though such a variety may have a toric structure for various tori. In fact, the choice of torus does not affect our analysis of toric varieties given below, and we refer interested readers to \cite{Duncan} for such considerations.

Recall that a $k$-\emph{form} of a $k$-variety $X$ is a $k$-variety $X'$ such that $X_L \simeq X_L '$ for some field extension $L/k$.  Any $k$-form of a toric variety is a toric variety \cite{Duncan}.
\end{rem}


\subsection{The split case}  Let us begin by recalling some facts concerning toric varieties with $T \simeq \gm ^n$ (e.g., when $k = \cplx$ or $k= k^s$), which are studied in terms of combinatorial data, e.g., lattices, cones, fans.  Good references for toric varieties over $\cplx$ include \cite{Fulton, CLS}, and many results hold generally in the split case.

Let $N$ be a finitely generated free abelian group of rank $n$ and $M = \hom (N, \Z)$.  A subsemigroup $\sigma \subset N_{\real}$ is a \emph{cone} if ($\sigma ^{\vee})^{\vee} = \sigma$, where $\sigma ^{\vee} = \{ u \in M \mid u(v) \geq 0 \text{ for all } v \in \sigma\}$.  A subsemigroup $\tau$ is a \emph{face} of $\sigma$ if it is of the form $\tau = \{v \in \sigma \mid u(v) = 0 \text{ for all } u \in S \}$ for some $S \subseteq \sigma ^{\vee}$.  A cone $\sigma$ is \emph{pointed} if 0 is a face of $\sigma$, and in this case $\sigma^{\vee}$ generates $M_{\real}$.  Given a pointed cone $\sigma$, we associate the affine $k$-scheme $U_{\sigma} = \Spec k[\sigma ^{\vee}]$, and for any face $\tau \subset \sigma$ the induced map $U_{\tau} \hookrightarrow U_{\sigma}$ is an open embedding.

A \emph{fan} $\Sigma \subset N_{\real}$ is a finite collection of pointed cones such that (1) any face of a cone in $\Sigma$ is a cone in $\Sigma$ and (2) the intersection of any two cones in $\Sigma$ is a face of each.  To any fan $\Sigma$ we associate a $k$-variety $X_{\Sigma}$ which is obtained by gluing the affine schemes $U_{\sigma}$ along common subschemes $U _{\tau}$ corresponding to faces.

On the other hand, beginning with a split torus $T \simeq \gm ^n$ and toric $T$-variety $X$ with fixed embedding $T \hookrightarrow X$, we recover $M$ as the character lattice $\hom (T, \gm)$ of $T$ and $N$ as the cocharacter lattice $\hom (\gm, T)$.  The association $\Sigma \mapsto X_{\Sigma}$ defines a bijective correspondence between fans $ \Sigma \subset N_{\real}$ and toric $T$-varieties $X$ (we remind the reader that here we assume $T$ is a split torus;  in general, fans $\Sigma$ admitting an action by $\text{Gal}(k^s/k)$ are in bijection with neutral toric $T$-varieties).

Let $\Sigma(\ell)$ denote the collection of cones in $\Sigma$ of dimension $\ell$.  Let $\text{Div}_T(X)$ denote the free abelian group generated by the \emph{rays} of $\Sigma$, i.e., elements of $\Sigma (1)$.  By the Orbit-Cone Correspondence \cite[Thm. 3.2.6]{CLS}, $\text{Div}_T(X)$ is isomorphic to the group of $T$-invariant Weil divisors of $X$.  For $X$ a (split) smooth projective toric variety, we have natural identifications $\text{Pic}(X) = \text{Pic} (X_{k^s}) = \text{Cl}(X_{k^s}) = \text{Cl}(X)$ which yield an exact sequence $$0 \to M \to \text{Div}_T(X) \to \text{Pic}(X) \to 0.$$  In particular, if $X$ is of dimension $n$ and $m$ is the number of rays in $\Sigma$, the Picard rank of $X$ is $\rho = m-n$.

\begin{defn}
A variety $X$ is \emph{Fano} (resp. \emph{weak Fano}) if its anticanoncial class $-K_X$ is ample (resp. nef and big).  If $X$ is a normal variety, a Cartier $D$ divisor on $X$ is \emph{nef} (``numerically effective" or ``numerically eventually free") if $D \cdot C \geq 0$ for every irreducible curve $C \subset X$.  A divisor $D$ is \emph{very ample} if $D$ is base point free and $\varphi_D : X \to \bbP(\Gamma(X, \O_X(D))^{\vee})$ is an embedding. A divisor $D$ is \emph{ample} if $\ell D$ is very ample for some $\ell \in \Z^+$.  A line bundle $\O_X(D)$ is nef or (very) ample if the corresponding divisor $D$ is nef or (very) ample.  A Cartier divisor is \emph{numerically trivial} if $D\cdot C =0$ for every irreducible complete curve $C \subset X$.  Let $N^1(X)$ be the quotient group of Cartier divisors by the subgroup of numerically trivial divisors.  The \emph{nef cone} $\text{Nef}(X)$ is the cone in $N^1(X)$ generated by the nef divisors, and the \emph{anti-nef cone} is the cone $-\text{Nef}(X) \subset N^1(X)$. A line bundle $\O_X(D)$ is nef (ample) if $D$ is nef (ample).  
\end{defn}

\begin{prop}\label{prop:nef}
A Cartier divisor $D$ on a split proper toric variety $X$ is nef (resp. ample) if and only if $D\cdot C \geq 0$ (resp. $D\cdot C > 0$) for all torus-invariant integral curves $C \subset X$.  
\end{prop}

\begin{proof}
When $k$ is algebraically closed, these are Theorems 3.1 and 3.2 of \cite{Mustata}. One can see that the arguments remain valid in the split case more generally.
\end{proof}


\subsection{The not necessarily split case}

Here we provide a ``black box''
for producing exceptional collections on arbitrary forms of toric
varieties by identifying certain special exceptional collections on
a \emph{split} toric variety.
This reduces an arithmetic question to a completely geometric question.

We begin by reviewing how to obtain arbitrary forms of toric varieties
from the split case
(see, for example, \cite{Vos82Projective,ELFST}).
Let $T$ be the split torus of a split smooth projective toric variety
$X$ with fan $\Sigma$ in the space $N \otimes \R$ associated to the
lattice $N$.
Let $\Aut(\Sigma)$ denote the subgroup of elements $g \in \operatorname{GL}(N)$
such that $g(\sigma) \in \Sigma$ for every cone $\sigma \in \Sigma$.
There is a natural inclusion of $T \rtimes \Aut(\Sigma)$
into $\Aut(X)$ as the subgroup leaving the open orbit $T$-invariant.

Let $k^s$ be the separable closure of $k$.
The Galois cohomology set $H^1(k^s/k,\Aut(X)(k^s))$ is in bijective
correspondence with the $k$-forms of $X$.
The natural map
\[ H^1(k^s/k, T(k^s) \rtimes \Aut(\Sigma)) \to H^1(k^s/k, \Aut(X)(L)) \]
in Galois cohomology is surjective;
the failure of this map to be a bijection amounts to the fact that there
may be several non-isomorphic toric variety structures
on the same variety (see \cite{Duncan} for more details).

Suppose $X'={}^\gamma X$ is a twisted form of a split toric variety
for a cocycle $\gamma$ representing a class in
$H^1(k^s/k, T(k^s) \rtimes \Aut(\Sigma))$.
There is a ``factorization'' $X'={}^\alpha({}^\beta X)$
where $\beta$ represents a class in $H^1(k^s/k, \Aut(\Sigma)$
and $\alpha$ represents a class in $H^1(k^s/k, ({}^\beta T)(k^s) )$.
Informally, $\beta$ changes the torus that acts on $X$,
while $\alpha$ changes the torsor of the open orbit in $X$.

Suppose $X$ is a toric $T$-variety.
We say that an object $E \in \mathsf{D^b}(X)$ is \emph{$T$-equivariant}
if $E$ is in the image of the forgetful functor from
$\mathsf{D^b}(\operatorname{Coh}_T(X))$
(see \S{2}~of~\cite{BFK2}).
In particular, this implies that $t^\ast E \simeq E$ for all $t \in
T(k)$.

\begin{prop} \label{prop:blackbox}
Let $X$ be a split toric $T$-variety over a field $k$ and let $\Sigma$
be the associated fan.  Suppose that $X$ admits an
$\Aut(\Sigma)$-stable full exceptional collection $\mathsf{E}$
such that each object is $T$-equivariant.
Then any $k$-form $X'$ of $X$ admits a full exceptional collection
$\mathsf{E'}$.
Moreover, $\mathsf{E'}$ is strong (resp. consists of vector bundles,
consists of sheaves) as soon as $\mathsf{E}$ is strong
(resp. consists of vector bundles, consists of sheaves).
\end{prop}

\begin{proof}
By Lemma~\ref{lem:Galconverse}, there exists a $G$-stable
exceptional collection $\mathsf{F}$ on $X_L$.
From the proof of that lemma, the objects $F$ of $\mathsf{F}$ are direct
summands of $\pi^\ast E$ for each object $E \in \mathsf{E}$,
where each isomorphism class of simple direct summand is represented
by exactly one $F$.
Since $\mathsf{E}$ is $\Aut(\Sigma)$-stable and each object
is $T$-equivariant, we may conclude that $\mathsf{F}$ is
$(T(L) \rtimes \Aut(\Sigma)) \rtimes G$-stable.

Let $X'$ be a $k$-form of $X$; there exists a finite Galois
extension $L/k$ with Galois group $G$ such that $X'_L \simeq X_L$.
From Theorem~5.1~of~\cite{Duncan}, the natural map
\[ H^1(L/k, T(L) \rtimes \Aut(\Sigma)) \to H^1(L/k, \Aut(X)(L)) \]
in Galois cohomology is surjective.
Thus, we may assume that $X' ={}^cX$ is the \emph{twist}
by a cocycle $c: G \to T(L) \rtimes \Aut(\Sigma)$.
Recall that the cocycle condition is that
$c(gh)=c(g){}^gc(h)$ for all $g,h \in G$
where ${}^gc(h)$ denotes the Galois action of $g$ on $T(L) \rtimes \Aut(\Sigma)$.

Identifying $X_L=X'_L$, twisting gives $\sigma'(g) = c(g) \sigma(g)$
where $\sigma$ is the action of $G$ induced from $X$ and
$\sigma'$ is induced from $X'$.
The punchline is that the action $\sigma'$ factors through the image of
$(T(L) \rtimes \Aut(\Sigma)) \rtimes G$ described above.
Thus the exceptional collection $\mathsf{F}$ is $G$-stable for the $X'$
action as well.
The proposition now follows by Theorem~\ref{thm:descblocks}.
\end{proof}

\begin{cor} \label{cor:toricLB}
Let $X$ be a split toric $T$-variety over a field $k$ and let $\Sigma$
be the associated fan.
If $X$ admits an $\Aut(\Sigma)$-stable full (strong) exceptional collection of
line bundles, then every $k$-form of $X$ admits a full (strong) exceptional
collection of vector bundles.
\end{cor}

\begin{proof}
Recall that every line bundle is isomorphic to a $T$-equivariant line
bundle by standard results on toric varieties.
The corollary now follows by Proposition~\ref{prop:blackbox}.
\end{proof}

\begin{lem}\label{lem:flips}
 Let $X$ and $Y$ be smooth projective toric varieties over $k$. Let $G = \operatorname{Gal}(k^s/k)$. Assume we have a $K$-positive toric flip $X \dashrightarrow Y$ such that over $k^{s}$ the flipping loci $F_i$ are disjoint and permuted by $G$. Let $H_i$ be the normalizer of $F_i$. If $X_L$ admits a full $G$-stable exceptional collection and $Y_i$ admits a full $H_i$-stable exceptional collection, then $Y$ admits a full exceptional collection.
\end{lem}

\begin{proof}
 Passing to $k^{s}$ we are free to use \cite{BFK} giving semi-orthogonal decompositions for the flip over each $Y_i$. Since the $Y_i$ are disjoint, we can concatenate these collections to get a $G$-stable collection. 
\end{proof}

\subsection{Products of toric varieties}

Recall that, given groups $G,H$ along with a homomorphism
$\rho : H \hookrightarrow S_n$,
the \emph{wreath product} $G \wr H$ is the group $G^n \rtimes H$
where $H$ acts on $G^n$ by permuting the copies of $G$.
We say a toric variety $X$ is \emph{indecomposable} if it cannot be written as
a product $X_1 \times X_2$ where $X_1$ and $X_2$ are
positive-dimensional toric varieties.

\begin{lem}\label{lem:autprod}
Suppose $Z= X_1^{n_1} \times \cdots \times X_r^{n_r}$ is a product of
proper split toric varieties $X_1, \ldots, X_r$, where $X_i \not\simeq X_j$
for $i \ne j$ and each $X_i$ is indecomposable.
Then
\[ \Aut(\Sigma) \simeq (\Aut({\Sigma_1}) \wr S_{n_1}) \times \cdots \times
(\Aut({\Sigma_r}) \wr S_{n_r}) , \]
where $\Sigma$ is the fan of $Z$ and $\Sigma_1, \ldots, \Sigma_r$ are
the fans of $X_1, \ldots, X_r$.
\end{lem}

\begin{proof}
First, consider $Z=X_1 \times X_2$ where
$X_1,X_2$ are proper split toric varieties.
Let $N$ (resp. $N_1, N_2$) be the cocharacter lattice and
$\Sigma$ (resp. $\Sigma_1, \Sigma_2$) be the fan of
$Z$ (resp. $X_1, X_2$).
Here $N = N_1 \oplus N_2$ and $\Sigma$ is the set of cones
of the form $\sigma_1 \times \sigma_2$ where $\sigma_1 \in \Sigma_1$
and $\sigma_2 \in \Sigma_2$.
The faces of a cone $\sigma_1 \times \sigma_2$ are precisely
the cones of the form $\sigma_1' \times \sigma_2'$ where
$\sigma_1'$ is a face of $\sigma_1$ and $\sigma_2'$ is a face of
$\sigma_2$.
The fan $\Sigma_1$ can be canonically identified with the subfan of
$\Sigma$ via the bijection $\sigma \mapsto \sigma \times \{0\}$.

Now, suppose also that $Z= Y \times W$ is a product of proper split
toric varieties where $Y$ is indecomposable.  Let $\Sigma_Y$ be the fan
of $Y$, which we can canonically identify with a subfan of $\Sigma_Z$.
Every cone of $Y$ is of the form $\sigma_1 \times \sigma_2$ where
$\sigma_1 \in \Sigma_1$ and $\sigma_2 \in \Sigma_2$.
Since fans are closed under taking faces, $\sigma_1 \times \{0\}$
and $\{0\} \times \sigma_2$ are also cones in $\Sigma_Y$.
Thus every cone in $\Sigma_Y$ is a product of cones in the intersections
$\Sigma_Y \cap \Sigma_1$ and $\Sigma_Y \cap \Sigma_2$.

In particular, since $X$ is proper, we have that the space
$N_Y \otimes \R$ is the direct sum of
$(N_Y \otimes \R) \cap (N_1 \otimes \R)$ and
$(N_Y \otimes \R) \cap (N_2 \otimes \R)$,
and $\Sigma_Y$ is a product of the fans $\Sigma_Y \cap \Sigma_1$
and $\Sigma_Y \cap \Sigma_2$.
Since $Y$ is indecomposable, one of these fans is indecomposable
and $\Sigma_Y$ must be a subfan of either $\Sigma_1$ or $\Sigma_2$.

Returning to the general case, we conclude that the decomposition
$\Sigma= \Sigma_1^{n_1} \times \cdots \times \Sigma_r^{n_r}$
is unique up to ordering.
The description of the automorphism group is immediate.
\end{proof}

\begin{lem}\label{lem:stabprod}
Let $Z$ be a proper toric $k$-variety with splitting field $L/k$.
Suppose $Z_L = \prod_{i=1}^n X_i$ where each $X_i$ is an indecomposable
split proper toric $L$-variety admitting a full (strong)
$\Aut(\Sigma_i)$-stable exceptional collection of line bundles,
where $\Sigma_i$ is the fan of $X_i$.
Then $Z$ has a full (strong) exceptional collection of vector bundles.
\end{lem}

\begin{proof}
It is a well known that the exterior product collection is an
exceptional collection.
For each isomorphism class among the $X_i$ fix a full (strong)
$\Aut(\Sigma_{X_i})$-stable exceptional collection of line bundles.
This ensures that the exterior product collection is
stable under the action of
$(\Aut(\Sigma_{X_1}) \wr S_{a_1}) \times \cdots \times (\Aut(\Sigma_{X_r})
\wr S_{a_r})$.
Since this group is $\Aut(\Sigma)$ by Lemma~\ref{lem:autprod},
the exterior product collection descends by Corollary~\ref{cor:toricLB}.
\end{proof}


\section{Low dimension or high symmetry}\label{section:minimal}

We provide exceptional collections for smooth toric surfaces, Fano 3-folds,
some Fano 4-folds, centrally-symmetric toric varieties, and
toric varieties corresponding to root systems of type $A$. 


\subsection{Surfaces}\label{sect:surfaces}
Here we prove that every toric surface has a full exceptional
collection.
We begin by recalling the (classical) minimal model program for surfaces
over non-closed fields.

Suppose $f : X \to X'$ is a birational morphism of smooth projective
surfaces over a field $k$.
If $k$ is separably closed, then by Proposition~5~of\cite{Coombes}
the morphism factors into a sequence
\[
X= X_0 \to X_1 \to \cdots \to X_r = X'
\]
where each morphism $X_i \to X_{i+1}$ is the blowup of a point on
$X_{i+1}$.
Over a non-closed field $k$, we can factor $f : X \to X'$ into a
sequence where each morphism $X_i \to X_{i+1}$ is defined over $k$
and is a blowup of a
(necessarily finite) Galois orbit of $k^s$-points on $X_{i+1}$.

Blowing up a point produces an exceptional curve: a smooth rational
curve with self-intersection $-1$.  By Castelnuovo's contractibility
criterion, such a curve can always be obtained as the result of a blow-up.
If one finds a skew Galois orbit of such curves on $X$, then there
exists a birational morphism $f : X \to X'$ contracting these curves.
Repetition of this procedure eventually terminates.

\begin{defn}
A \emph{minimal surface} $X$ is a smooth projective surface over a field $k$
such that every birational morphism $X \to X'$ to a smooth projective
surface $X'$ is an isomorphism.
\end{defn}

Any smooth projective surface can be obtained by iteratively blowing up
Galois orbits of separable points starting from a minimal model.
A toric variety is geometrically rational.
Minimal geometrically rational surfaces were classified by
Manin~\cite{Manin} and Iskovskikh~\cite{Iskovskikh}.
One checks that the toric surfaces in their collection are the
following (see also a direct proof in~\cite{Xie}):

\begin{lem}\label{lem:classification}
A minimal smooth projective toric surface
is a $k^s/k$-form of one of the following:
\begin{enumerate}
\item $\mathbb{P}^2$, $\Aut(\Sigma) = S_3$.
\item $\pone \times \pone$, $\Aut(\Sigma) = D_8$.
\item $\mathbb{F}_{a} = \operatorname{Proj}(\O_{\pone} \oplus \O_{\pone} (a))$,
$a \geq 2$, $\Aut(\Sigma) = C_2$.
\item $\mathsf{dP}_6 = $ del Pezzo surface of degree 6, $\Aut(\Sigma) = D_{12}$.
\end{enumerate}
\end{lem}

\begin{proof}
A minimal geometrically rational surface is either a del Pezzo surface
or has a conic bundle structure \cite{Manin,Iskovskikh}.
Over the separable closure, a del Pezzo surface is either
$\pone \times \pone$ or a blow up of $\PP^2$ at up to $8$ points in
general position.
Blowing up only one or two points never results in a minimal surface, and
no more than three points can be simultaneously torus invariant and
in general position.  Thus every del Pezzo surface is a $k^s/k$-form
of $\PP^2$, $\PP^1 \times \PP^1$ or $\dP_6$.
Over the separable closure, a conic bundle structure has at most $2$
singular fibers since their images must be torus invariant points on the
base $\PP^1$.
A minimal conic bundle with at most two singular fibers over the
separable closure must be either a del Pezzo surface or a minimal ruled
surface.
\end{proof}

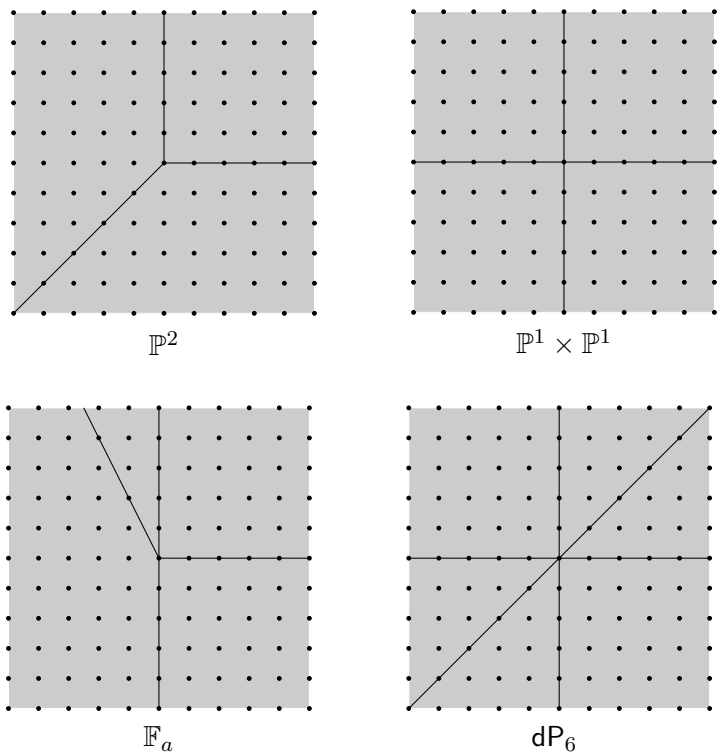
\begin{figure}
 \begin{center}
 \begin{tikzpicture}
   [scale=.4, vertex/.style={circle,draw=black!100,fill=black!100, inner sep=0.5pt,minimum size=0.5mm}]
   \filldraw[fill=black!20!white,draw=white!100]
     (-5,5) -- (5,5) -- (5,-5) -- (-5,-5) -- (-5,5);
   \draw (0,0) -- (5,0);
   \draw (0,0) -- (0,5);
   \draw (-5,-5) -- (0,0);
  \foreach \x in {-5,-4,...,5}
   \foreach \y in {-5,-4,...,5}
   {
     \node[vertex] at (\x,\y) {};
   }
\node at (0,-6,0) {\text{$\mathbb{P}^2$}};
 \end{tikzpicture}
\hspace{1cm}
 \begin{tikzpicture}
   [scale=.4, vertex/.style={circle,draw=black!100,fill=black!100, inner sep=0.5pt,minimum size=0.5mm}]
   \filldraw[fill=black!20!white,draw=white!100]
     (-5,5) -- (5,5) -- (5,-5) -- (-5,-5) -- (-5,5);
   \draw (0,0) -- (5,0);
   \draw (0,0) -- (0,5);
   \draw (0,0) -- (0,-5);
   \draw (-5,0) -- (0,0);
  \foreach \x in {-5,-4,...,5}
   \foreach \y in {-5,-4,...,5}
   {
     \node[vertex] at (\x,\y) {};
   }
\node at (0,-6,0) {\text{$\pone \times \pone$}};
 \end{tikzpicture}

\vspace{0.5cm}

 \begin{tikzpicture}
   [scale=.4, vertex/.style={circle,draw=black!100,fill=black!100, inner sep=0.5pt,minimum size=0.5mm}]
   \filldraw[fill=black!20!white,draw=white!100]
     (-5,5) -- (5,5) -- (5,-5) -- (-5,-5) -- (-5,5);
   \draw (0,0) -- (5,0);
   \draw (0,0) -- (0,5);
   \draw (0,0) -- (0,-5);
   \draw (-2.5,5) -- (0,0);
  \foreach \x in {-5,-4,...,5}
   \foreach \y in {-5,-4,...,5}
   {
     \node[vertex] at (\x,\y) {};
   }
\node at (0,-6,0) {\text{$\mathbb{F}_a$}};
 \end{tikzpicture}
\hspace{1cm}
 \begin{tikzpicture}
   [scale=.4, vertex/.style={circle,draw=black!100,fill=black!100, inner sep=0.5pt,minimum size=0.5mm}]
   \filldraw[fill=black!20!white,draw=white!100]
     (-5,5) -- (5,5) -- (5,-5) -- (-5,-5) -- (-5,5);
   \draw (0,0) -- (5,0);
   \draw (0,0) -- (0,5);
   \draw (0,0) -- (0,-5);
   \draw (0,0) -- (-5,0);
   \draw (5,5) -- (0,0);
   \draw (-5,-5) -- (0,0);
  \foreach \x in {-5,-4,...,5}
   \foreach \y in {-5,-4,...,5}
   {
     \node[vertex] at (\x,\y) {};
   }
\node at (0,-6,0) {\text{$\mathsf{dP}_6$ }};
 \end{tikzpicture}
\vspace{-.3cm}
\caption{Fans for minimal toric surfaces}\label{fig:fans}
\end{center}
\end{figure}

Here we exhibit full strong exceptional collections consisting of
$G$-stable blocks for each minimal toric surface exhibited above
(none of these collections are original).
The fans associated to the split forms of these surfaces are given in Figure~\ref{fig:fans}. In each case, we fix a torus $T$ which gives $X$ the structure of a toric $T$-surface.  As remarked above, this gives a homomorphism $G \to \text{Aut}(\Sigma)$ as well as an action of $G$ on $\text{Pic}(X_L)$, where $L$ is a splitting field of $T$, $G = \text{Gal}(L/k)$, and $\Sigma$ is the fan corresponding to the split toric surface $X_L$. We produce $G$-stable exceptional collections in each case by exhibiting $\text{Aut}(\Sigma)$-stable collections.

\begin{ex}\label{ex:p2}
Let $X$ be a toric $T$-surface whose split form is $\mathbb{P}^2$ with
$\text{Aut}(\Sigma) = S_3$. The $ S_3$-action on
$\text{Pic}(\mathbb{P}^2) = \Z$ is clearly trivial, so that the
exceptional collection $\{ \O, \O(1), \O(2)\}$, given in
\cite{Beilinson} yields a full strong $\text{Aut}(\Sigma)$-stable
exceptional collection.  By Corollary~\ref{cor:toricLB},
$X$ admits a full strong exceptional collection.
\end{ex}

\begin{ex}\label{ex:p1p1}
Let $X$ be a toric surface whose split form is $\pone \times \pone$ with $\text{Aut}(\Sigma) = D_8$, and consider the natural projections $p_1, p_2: \pone \times \pone \to \pone$.  Let $\O(p, q) = p_1 ^*\O (p) \otimes p_2^* \O (q)$.  By \cite{KvichanskyNogin}, the collection $\{\O, \O(1, 0), \O(0, 1), \O(1,1) \}$ on $\pone \times \pone$ is exceptional since $\{ \O, \O(1)\} $ is an exceptional collection for $\pone$. The $ D_8$-action preserves this collection, with orbits given by the blocks $\mathsf{E}^0 = \{ \O \}$, $\mathsf{E}^1 = \{ \O(1, 0), \O(0,1) \} $, and $\mathsf{E}^2= \{ \O(1,1)\}$.  In particular, this collection above is $\text{Aut}(\Sigma)$-stable, and Corollary~\ref{cor:toricLB} yields an exceptional collection on $X$.
\end{ex}

\begin{ex}
Let $X$ be a toric surface whose split form is the Hirzebruch surface
$\mathbb{F}_a$; here $\text{Aut}(\Sigma) = C_2$.
Let $e_1, e_2$ be the standard basis for $\Z^2$.
As in \cite[Ex. 4.1.8]{CLS},
let $u_1  =-e_1 + ae_2$, $u_2 = e_2$, $u_3 = e_1$, and $u_4
= -e_2$ be the generators of $\Sigma(1)$ with corresponding toric
divisors $D_i$.  The Picard group of $\mathbb{F}_a$ is freely generated
by $\{D_1, D_2\}$ and $D_1$ is linearly equivalent to $D_3$.
The only nontrivial fan automorphism $\sigma$ takes $e_1 \mapsto -e_1+ae_2$
and $e_2 \mapsto e_2$.
Thus $\sigma$ leaves $D_2,D_4$ fixed and interchanges $D_1$ and $D_3$.
We conclude the action of $C_2$ on
$\text{Pic}(\mathbb{F}_a)$ is trivial, and thus, any exceptional
collection is necessarily $G$-stable (see Lemma~\ref{lem:picinv}).
An exceptional collection for $\mathbb{F}_a$ is given by $\{\O, \O(D_3),
\O(D_4), \O(D_3 + D_4)\}$ \cite{KvichanskyNogin}.
Corollary~\ref{cor:toricLB} then gives an exceptional collection on $X$.
\end{ex}

\begin{ex}\label{ex:dp6}
Let $X$ be a toric surface whose split form is $\mathsf{dP}_6$; here
$\text{Aut}(\Sigma) = D_{12}$. Viewing $\mathsf{dP}_6$ as the blowup of
$\bbP ^2$ at 3 non-colinear points, let $H$ be the pullback of the
hyperplane divisor on $\bbP^2$ and $E_i$ the exceptional divisors, $i =
1, 2, 3$.  As shown in \cite[Prop. 6.2(ii)]{King}, the collection
$$\{\O, \O(H - E_1), \O(H - E_2), \O(H - E_3), \O(H), \O(2H - (E_1 + E_2
+ E_3)) \}$$ gives an exceptional collection for $\mathsf{dP}_6$,
which is $\text{Aut}(\Sigma)$-stable.

Let us rephrase this in the notation of \cite{BSS}.
There are two morphisms $\mathsf{dP}_6 \to \bbP^2$ realizing
$\mathsf{dP}_2$ as a blowup of $\bbP^2$, and we denote the collection of
all six exceptional divisors by $L_i$ and $M_i$, with $i = 1, 2, 3$.
Let $H$ and $H'$ denote the pullbacks of the hyperplane divisors on
$\bbP^2$ under the maps contracting $M_i$ and $L_i$, respectively, where
we identify $H$ with the divisor given in King's collection above (and
thus we also identify $E_i$ with $M_i$).
Then $H = L_1 + M_2 + M_3$, and it follows that $$2H - (E_1 + E_2 + E_3)
= L_1 + L_2 + M_3 = H'$$ using the relation $L_i + M_j = L_j + M_i$.
Furthermore, one checks that $ H - E_1 =  L_2 + M_3$, $H-E_2 = L_1 +
M_3$, and $H- E_3  = L_1 + M_2$.
As described in \cite[$\S$2]{BSS}, the element $\sigma$ in  $S_3 \times
C_2 = D_{12}$ which cyclically permutes the six lines $L_i, M_i$ also
satisfies $\sigma (H) = H'$ and $\sigma^2(H) = H$.
We arrange the exceptional collection above into blocks $\mathsf{E}^0 =
\{\O \}$, $\mathsf{E}^1 = \{ \O(H - E_1), \O(H - E_2), \O(H - E_3)\}$
and $\mathsf{E}^2 = \{\O(H), \O(2H - (E_1 + E_2 + E_3))\}$.
In particular, the exceptional collection given above is
$\text{Aut}(\Sigma)$-stable, and by Corollary~\ref{cor:toricLB} we have an exceptional
collection on $X$.
\end{ex}

\begin{prop}\label{prop:surface}
Every toric surface admits a full exceptional collection of sheaves.
\end{prop}

\begin{proof}
There is a sequence of blowups $X = X_0 \to \cdots \to X_s = X'$ where
$X'$ is minimal, so must be one of the varieties given in
Lemma~\ref{lem:classification}.
By Examples~\ref{ex:p2}-\ref{ex:dp6},
$X'$ admits a full strong exceptional collection of vector bundles, and
thus $X'_L$ admits a $G$-stable exceptional collection.  By
Lemma~\ref{lem:blowup}, $X_L$ admits a $G$-stable exceptional
collection.
\end{proof}

\begin{rem}
The authors would like to thank F.~Xie for pointing out a mistake
in the statement of a previous version of
Proposition~\ref{prop:surface}.
Xie also discusses exceptional collections of toric surfaces in
\cite{Xie}, although her definition of exceptional object is not the
same as ours. 
In the second arXiv version of that paper, Xie sketched in Remark~8.8
how one might construct an exceptional collection for toric surfaces.
After the authors posted a preliminary version of this paper to the
arXiv, Xie updated her preprint with Corollary~8.8, which proves the analog of
the above proposition for collections of vector bundles but using her notion of exceptional collection.
\end{rem}


\subsection{The toric Frobenius and toric Fano 3-folds}\label{sect:3fold}
In Table~\ref{tab:3-folds} we present the classification of smooth toric Fano 3-folds given in \cite{Batyrev, Watanabe}, adopting Batyrev's enumeration. For each $X = X_{\Sigma}$, we record the following invariants:
\begin{itemize}

\item $\sigma(1) = | \Sigma (1) |$ is the number of rays of $\Sigma$ \cite{BT}.

\item $k_0$ is the rank of the Grothendieck group $K_0(X)$, which coincides with the number of maximal cones in the fan $\Sigma$ \cite{BT}.

\item $\Aut(\Sigma) $ is the automorphism group of the (lattice $N$ which preserves the) fan $\Sigma$ corresponding to $X$.

\item $\rho$ is the Picard rank of $X$ \cite{Watanabe}.

\item $\rho ^G$ is the $\Aut(\Sigma)$-invariant Picard rank of $X$,
i.e., the rank of $\text{Pic}(X)^{\Aut(\Sigma)}$.

\item $\mathfrak{fr} = | \mathsf{Frob}(X) |$ is the number of
isomorphism classes of line bundles produced by the push forward of the structure sheaf under the Frobenius morphism \cite{BT, Uehara}.

\item $ \mathfrak{fr}^- = |\mathsf{Frob}(X) \cap -\text{Nef}(X)|$ is the
number of isomorphism classes of line bundles in $\mathsf{Frob}(X)$ which lie in the anti-nef cone of $X$ \cite{Uehara}.
\end{itemize}

\begin{table}
\begin{center}
\begin{tabular}{lrlccccccc}
\toprule
  & &Toric Fano 3-fold $X$ & $\sigma(1)$ &$k_0$ & $\operatorname{Aut}(\Sigma)$ & $\rho$ & $\rho^G$ & $\mathfrak{fr}$ & $\mathfrak{fr}^- $\\
\midrule

& 1. & $\bbP^3$ & 4 &  4 & $S_4$ & 1 & 1 & 4 & 4  \\

&2.  &$\bbP_{\bbP^2}(\O \oplus \O(2))$ & 5 & 6 & $S_3$ & 2 & 2 & 7 & 6 \\

 &3.  &$\bbP_{\bbP^2}(\O \oplus \O(1))$ & 5 & 6 & $S_3$ & 2 & 2 & 6 & 6\\
 
&4.& $\bbP_{\bbP^1}(\O \oplus \O \oplus \O(1))$ & 5 & 6 & $C_2 \times C_2$ & 2 & 2 & 6 & 6\\

 &5. & $\bbP^2\times \bbP^1$ & 5 & 6 & $D_{12}$ & 2 & 2 & 6 & 6\\
 
 &6.  &$\bbP_{\bbP^1\times \bbP^1}(\O\oplus \O(1,1))$ & 6 & 8 & $D_8$ & 3 & 2 & 8 & 8\\
 
 &7.  &$\bbP_{\dP_8}(\O\oplus \O(l))$, $l^2=1$ on $\dP_8$ & 6 & 8 & $D_8$ & 3 & 3 & 8 & 8 \\

&8.  &$\bbP^1\times \bbP^1 \times \bbP^1$ & 6 & 8 & $C_2 \times S_4$ & 3 & 1 & 8 & 8\\

 &9.  &$\dP_8\times \bbP^1 $ & 6 & 8 &$ C_2 \times C_2$ & 3 & 3 & 8& 8 \\
 
&10. &$\bbP_{\bbP^1\times\bbP^1}(\O\otimes \O(1,-1))$& 6 & 8 & $D_8$ & 3 & 2 & 8 & 8\\

 &11.  & $\text{Bl}_{\bbP^1}(\bbP_{\bbP^2}(\O \oplus \O(1)))$& 6 & 8 & $C_2$ & 3 & 3 & 9 & 8 \\
 
 &12. & $\text{Bl}_{\bbP^1}(\bbP^2\times \bbP^1)$& 6 & 8 & $C_2$ & 3 & 3 & 8 & 8 \\

& 13.  & $\dP_7-$bundle over $\bbP^1$ & 7 & 10 & $C_2$ & 4 & 4 & 10 & 10 \\

 & 14. & $\dP_7-$bundle over $\bbP^1$ & 7 & 10 & $C_2 \times C_2$ & 4 & 3 & 10 & 10\\

& 15. & $\dP_7\times \bbP^1$& 7 & 10 & $C_2 \times C_2$ & 4 & 3 & 10 & 10\\

& 16. & $\dP_7-$bundle over $\bbP^1$& 7 & 10 & $C_2$ & 4 & 4 & 10 & 10\\

 & 17.  &$\dP_6\times \bbP^1$& 8 & 12 & $C_2 \times C_2 \times S_3$ & 5 &  2 & 12 & 12\\

 & 18. & $\dP_6-$bundle over $\bbP^1$& 8 & 12 & $C_2 \times C_2$  & 5 & 4 & 12 & 12\\
\bottomrule
\end{tabular}
\vspace{.2cm}
\caption{Toric Fano 3-folds}\label{tab:3-folds}
\end{center}
\end{table}

\subsubsection{Toric Frobenius}\label{section:frob}

Let $X$ be a split toric variety of dimension $n$ with fixed torus embedding $T \hookrightarrow X$ and take $\ell \in \Z^+$.  Define the $\ell^{\text{th}}$ Frobenius map on $T = \gm^n$ to be $(x_1,..., x_n) \mapsto (x_1^{\ell},..., x_n^{\ell})$.  The unique extension to $X$ will be denoted $F_{\ell}$ and called the \emph{$\ell^{th}$ Frobenius morphism}.  Alternatively, if $\Sigma \subset N$ is the fan associated to $X$, define a lattice $N' = \frac{1}{\ell} N$.  The inclusion $N \subset N'$, which sends a cone in $N_{\real}$ to the cone with the same support in $N'_{\real}$, induces a finite surjective morphism which is precisely the $\ell^{\text{th}}$ Frobenius morphism $F_{\ell}: X \to X.$ 

The sheaf $(F_{\ell})_*(\O_X)$ splits into line bundles and Thomsen
provides an algorithm for computing its direct summands  \cite{Thomsen}.
We let $\mathsf{Frob}(X)$ denote the union of all isomorphism classes of line bundles arising as direct summands of $(F_{\ell})_*(\O_X)$ as $\ell$ varies over $\Z^+$.  Note that $\mathsf{Frob}(X)$ is a finite set. 

\begin{conj}[Bondal \cite{Bondal2}]\label{conj:Bondal}
If $X$ is a smooth proper toric variety then the collection $\mathsf{Frob}(X)$ generates $\mathsf{D^b}(X)$.
\end{conj}

For a toric variety $X$ in which Bondal's Conjecture is true, we will
say that \emph{the Frobenius generates the derived category of} $X$.
In loc. cit., Bondal proves that if all summands of $\mathsf{Frob}(X)$ are nef, one actually gets a full strong exceptional collection, so that Conjecture~\ref{conj:Bondal} is true in this case. He also notes his arguments work for all but two (isomorphism classes of) toric Fano threefolds. To cover all toric Fano threefolds, 
Uehara noticed that discarding line bundles which do
not lie in the set $-\text{Nef}(X)$ yields a full strong
exceptional collection \cite{Uehara}.

\begin{lem}\label{lem:FrobNef}
Let $X$ be a toric variety over $k$ with splitting field $L$.
Suppose $\mathsf{E}$ is a full (strong) exceptional collection for
$\mathsf{D^b}(X_L)$
where either $\mathsf{E} = \operatorname{\mathsf{Frob}}(X_L)$ or
$\mathsf{E} = \operatorname{\mathsf{Frob}}(X_L) \cap - \operatorname{Nef}(X_L)$.
Then there exists a full (strong) exceptional collection for
$\mathsf{D^b}(X)$.
\end{lem}

\begin{proof}
Both $\operatorname{\mathsf{Frob}}(X_L)$ and $\text{Nef}(X_L)$ are
canonical constructions and thus are $\text{Aut}(X_L)$-stable.
In particular, $\mathsf{E}$ is $\text{Aut}(\Sigma)$-stable
and so Corollary~\ref{cor:toricLB} applies.
\end{proof}

\begin{prop}\label{prop:3fold}
Let $X$ be a smooth projective toric Fano 3-fold over a field $k$.  Then $X$ admits a full strong exceptional collection consisting of vector bundles.
\end{prop}

\begin{proof}
Let $X_L$ be the associated split toric Fano 3-fold.  The main result of \cite{Uehara} guarantees that the set $\mathsf{E} = \mathsf{Frob}(X_L) \cap - \text{Nef}(X_L)$ defines a full strong exceptional collection on $X$. Lemma~\ref{lem:FrobNef} completes the proof.
\end{proof}


\subsection{Toric Fano 4-folds}\label{sect:4fold}

There are 124 split smooth toric Fano 4-folds,
which were first classified in \cite{Batyrev}
(a missing case was added in~\cite{Sato}).
In \cite{Prabhu}, Prabhu-Naik exhibits full strong exceptional
collections for all 124 of these 4-folds.
However, it is not clear that these collections are
$\Aut(\Sigma)$-stable, so they do not necessarily lead to full strong
exceptional collections in the arithmetic case.

All collections obtained
using Method 1 of loc. cit. produce $\Aut(\Sigma)$-stable collections (note that
this is precisely the method used in \cite{Uehara} for toric Fano
3-folds, and we will refer to this as the \emph{Bondal-Uehara Method}).
Together with Lemmas~\ref{lem:stabprod} and \ref{lem:FrobNef}, this gives stable exceptional collections for 43 of the 124 smooth toric Fano 4-folds.  However, there are examples when the
Bondal-Uehara Method fails to produce an exceptional collection.
In this case, all is not lost (see Section~\ref{sec:centsym}).

More precisely, the varieties (61), (62), (63), (64), (77), (105),
(107), (108), (110), (122), and (123) of \cite{Prabhu} are shown to have
exceptional collections using the Bondal-Uehara Method.
Hence, they admit exceptional collections which are $\Aut(\Sigma)$-stable
and thus provide exceptional collections for the arithmetic forms. Secondly, for the varieties (109), (114), and (115), the set
$\mathsf{Frob}(X)$ is a full exceptional collection, which is $G$-stable
by Lemma~\ref{lem:FrobNef}. Lastly, Lemma~\ref{lem:stabprod} guarantees the existence of exceptional collections on products.  Hence, the following varieties also admit stable exceptional collections: (0), (4), (9), (17), (24), (25), (26), (27), (45), (52), (53), (54), (55), (56), (58), (67), (73), (88), (90), (92),
(93), (97), (103), (111), (112), (113), (118), (119), (120).


\subsection{Centrally symmetric toric Fano varieties} \label{sec:centsym}

Polytopes with the highest degree of symmetry are the \emph{centrally
symmetric} polytopes, i.e., $-P = P$.
The smooth split toric varieties $X$ whose anti-canonical polytope is full-dimensional and centrally symmetric were classified in \cite{VosKly}. It was shown that any such variety (which we refer to as a \emph{centrally symmetric toric Fano varieties}) is isomorphic to a product of projective lines and \emph{generalized del Pezzo varieties} $V_n$ of dimension $n = 2m$.  Note that $V_2 = \mathsf{dP}_6$ and $V_4$ is the missing (116)
from the list in Section~\ref{sect:4fold} (this is (118) in the enumeration found in \cite{Batyrev}). The goal of this section is to exhibit full stable exceptional collections on $V_n$, which in turn yields stable exceptional collections for any centrally symmetric toric Fano variety, in light of Lemma~\ref{lem:stabprod}.

In \cite[Theorem 6.6]{CT}, Castravet and Tevelev found $\Aut(\Sigma)$-stable full
strong exceptional collections for the varieties $V_n$.
The authors of this paper had independently discovered the same
exceptional collection (up to a twist by a line bundle).
Nevertheless, the perspective here may be of independent interest, so we
sketch the argument. The authors give a more detailed analysis in \cite{BDMdP}.

The variety $V_{n}$ with $n= 2m$ has rays given by
$$\begin{array}{rl}
e_1 &= (1,0,\cdots,0)\\
e_2 &= (0,1,\cdots,0)\\
&\vdots\\
e_n &= (0,0,\cdots,1)\\
e_{n+1} &= (-1,-1,\cdots,-1)\\
\end{array}
\hspace{.4cm}
\begin{array}{rl}
\bar{e}_1 &= (-1,0,\cdots,0)\\
\bar{e}_2 &= (0,-1,\cdots,0)\\
&\vdots\\
\bar{e}_n &= (0,0,\cdots,-1)\\
\bar{e}_{n+1} &= (1,1,\cdots,1)
\end{array}$$ and whose maximal cones are given as follows.  Among the rays $e_1,..., e_{n+1}$, omit a single $e_i$.  From the remaining $n = 2m$ rays, choose $\frac n 2$ of them and take their antipodes \cite[Proof of Thm. 5]{VosKly}. Note that $V_2 = \mathsf{dP}_6$ (whose fan is given in Figure~\ref{fig:fans}).  The number of maximal cones $c(n)$ of $V_n$ is given by
\[
c(n)=\frac{(n+1)!}{(\frac n 2)!^2} = \frac{(2m +1)! }{m!^2}\ .
\]
There's a natural action of $S_{n+1} \times C_2$, where $S_{n+1}$
permutes $e_1,\ldots,e_{n+1}$ and $\bar{e}_1\ldots\bar{e}_{n+1}$ in the
obvious way.  The $C_2$-action is simply the antipodal map on the
cocharacter lattice --- we will refer to it as ``the involution.''
Clearly, the involution interchanges $e_i$ and $\bar{e_i}$. 

The variety $V_n$ is of importance in birational geometry due its
appearance in the factorization of the standard Cremona transformation
of $\bbP^n$. In fact, as is well-known, $V_n$ can be explicitly obtained
from $\bbP^n$ as follows. First blow up the torus fixed points, then
flip the (strict transforms) of the lines through these points, then
flip the (strict transforms) of planes through these points, \ldots, up
until, and not including, the half-dimensional linear subspaces. The
resulting variety is $V_n$. For more, see \cite{Casagrande}. 

Since $V_n$ and the blow up of $\bbP^n$ at its torus fixed points are
isomorphic in codimension $1$, they have isomorphic Picard groups. We
use a basis $H,E_1, \ldots, E_{n+1}$ for $\Pic(V_n)$,
which correspond to the hyperplane section and the
exceptional divisors of the blown up $\bbP^n$.
We have
\[
[e_i] = E_i, \quad
[\bar{e}_i] = (H-\sum_{j=1}^{n+1}E_j)+E_j
\]
where $S_{n+1}$ permutes the $E_i$ leaving $H$ fixed, and the involution
is represented by the following matrix
\[
\begin{pmatrix}
n & 1 & 1 & \cdots & 1 \\
1-n & 0 & -1 & \cdots & -1 \\
1-n & -1 & 0 & \cdots & -1 \\
\vdots & \vdots & \vdots & \ddots & \vdots\\
1-n & -1 & -1 & \cdots & 0 \\
\end{pmatrix} \ .
\]

For each $c \in \Z$ and $J \subset \{1, \ldots, n+1\}$, define 
\[
 F_{c,J} := c\left(\sum_{i=1}^{n+1}E_i - H\right) - \sum_{j \in J} E_j .
\]
Note that the involution takes $F_{c,J}$ to $F_{|J|-c,J}$. Then, 

\begin{prop} \label{prop:Vn}
 The set of $F_{c,J}$ with
 \begin{enumerate}
  \item $\displaystyle{|J|-\frac{n}{4} \le c \le \frac{n}{4}}$, or
  \item $\displaystyle{\frac{n+2}{4} \le c \le |J| - \frac{n+2}{4}}$.
 \end{enumerate}
form a full strong $(S_{n+1} \times C_2)$-stable exceptional collection
on $V_n$ under any ordering of the
blocks such that $|J|$ is (non-strictly) decreasing.
\end{prop}

\begin{sproof}
This collection is the same as that of \cite[Theorem 6.6]{CT} up to a
twist by a line bundle.
Thus, we only sketch an argument here (expanded in \cite{BDMdP}).
One checks that the description of ``forbidden cones'' given by
Borisov and Hua in \cite{BH} shows that relevant cohomology groups
vanish --- this shows that it is a strong exceptional collection.
To prove generation, one considers the series of flips required to reach
$\bbP^n$ blown up at $n+1$ points.
Using the description of the semi-orthogonal decompositions in
\cite{BFK}, the line bundles can be shown to generate the necessary
admissible subcategories of each intermediate birational model.
\end{sproof}

Since any centrally symmetric toric Fano variety is a product of projective lines and the varieties $V_n$, Lemma~\ref{lem:stabprod} yields the following:

\begin{cor}\label{cor:centsym}
Any form of a centrally symmetric toric Fano variety admits a full
strong exceptional collection consisting of vector bundles.
\end{cor}

\subsection{Toric varieties from the Weyl fans of type A} \label{sec:X(An)}

One method for identifying toric varieties with large symmetry groups is to start with root systems. Let $R$ be a root system in a Euclidean space $E$. The $\Z$-lattice generated by $R$ is denoted $M(R)$, while its dual in $E^\vee$ is denoted by $N(R)$. For every set $S$ of simple roots in $E$, we have the dual cone corresponding to a closed Weyl chamber
\begin{displaymath}
 \sigma_S := \{ f \in E^\vee \mid \langle f, \alpha \rangle \geq 0 \ , \ \forall \alpha \in S\}.
\end{displaymath}
The cones $\sigma_S$ are the maximal cones for a fan $\Sigma_R$ in $E^\vee$. We denote the associated toric variety by $X(R)$. Recall that an automorphism of $R$ is an element of $\operatorname{GL}(E)$ preserving $R$. Let $W(R)$ be the Weyl group and $\Gamma(R)$ the symmetry group of the Dynkin diagram of $R$. It is well-known that
\begin{displaymath}
 \text{Aut}(R) \simeq W(R) \rtimes \Gamma(R).
\end{displaymath}
Any automorphism of $R$ induces an action on the fan $\Sigma(R)$, which yields a homomorphism $\phi: \text{Aut}(R) \to \text{Aut}(\Sigma(R))$.

\begin{lem} \label{lem:X(R)FanAut}
 The map $\phi:  \operatorname{Aut}(R) \to \operatorname{Aut}(\Sigma(R))$ is an isomorphism.
\end{lem}

\begin{proof}
 First note that the set $R$ can be reconstructed from $\Sigma(R)$ by taking the union of the extremal rays generating the dual cones $\sigma_S^\vee$ for all $\sigma_S$. Thus any symmetry of the fan induces a symmetry of $R$. This gives the inverse map to $\phi$.
\end{proof}

Here we focus on the case $R = A_n$. In \cite{LosevManin}, the authors showed that $X(A_n)$ is a moduli space of rational curves with $(n+1)$ marked points and $2$ poles. Another useful proof appeared in \cite{BatyrevBlume}. 

Using this perspective, Castravet and Tevelev exhibited an exceptional
collection on $X(A_n)$ that is stable under the action of permuting the
marked points and flipping the poles, i.e., an $(S_{n+1} \rtimes
C_2)$-stable collection. Here we demonstrate that Castravet and
Tevelev's exceptional collection satisfies the conditions of
Proposition~\ref{prop:blackbox} and hence descends to an exceptional
collection on any form of $X(A_n)$ (in characteristic $0$). 

To do this requires a bit of translating divisors and actions from the moduli-theoretic language to the toric language. We recall the moduli-theoretic languge. 

\begin{defn}
 Let $N$ be a set of order $n$. A \emph{chain of polar} $\pone$'s is a $\left(\{0,\infty\} \cup N\right)$-marked linear nodal chain of $\mathbb{P}^1$'s with $0$ on the left tail and $\infty$ on the right tail. A chain of polar $\pone$'s is \emph{stable} if 
 \begin{enumerate}
  \item marked points do not coincide with nodes,
  \item only $N$-marked points are allowed to coincide,
  \item each component of the chain has at least three special points (nodes or marked points). 
 \end{enumerate}
 We write $LM_N$ for the corresponding moduli space. We also use $LM_n$ depending on the context.  Note that the universal curve over $LM_n$ is isomorphic to $LM_{n+1}$.
\end{defn}

\begin{thm}\label{thm:univcurve}
 The toric variety $X(A_{n-1})$ is isomorphic to $LM_n$. Moreover, if we fix an embedding $A_{n-1} \to A_n$, the corresponding map $X(A_n) \to X(A_{n-1})$ is the universal curve.
Moreover, $X(A_n) \to X(A_{n-1})$ is a toric morphism.
\end{thm}

\begin{proof}
This is \cite[Theorem 2.6.3]{LosevManin}. See also \cite[Theorem 3.19]{BatyrevBlume}. 
The map is consequently toric by \cite[Proposition~1.4]{BatyrevBlume}.
\end{proof}

Under this isomorphism, the closures of the torus orbits on $X(A_n)$ have the
following moduli-theoretic description.
Fix a partition $N_1 \sqcup N_2 = N$ and let $\delta_{N_1}$ denote the
divisor parametrizing polar chains of length exactly $2$ having the
first marked by $N_1$ and the last marked by $N_2$.
For a partition with more parts, $N_1 \sqcup N_2 \sqcup \cdots \sqcup N_t = N$,
one has the locus $Z_{N_1,\dots,N_t}$ parametrizing polar chains of length exactly $t$,
where the $i$-th $\pone$ is marked by $N_i$.
These loci are precisely the proper torus orbit closures on $X(A_n)$.

Note that each loci is a complete intersection
\begin{displaymath}
 Z_{N_1,\dots,N_t} := \delta_{N_1} \cap \delta_{N_1 \cup N_2} \cap \cdots \cap \delta_{N_1 \cup \cdots \cup N_{t-1}}. 
\end{displaymath}
Moreover, we have an isomorphism 
\begin{displaymath}
 Z_{N_1,\dots,N_t} \simeq LM_{N_1} \times LM_{N_2} \times \cdots \times LM_{N_t}
\end{displaymath}
where the left node of each $\pone$ is marked with $0$ and the right node is marked with $\infty$.
Thus, we have toric morphisms 
\begin{displaymath}
  i_{N_1,\ldots,N_t} : LM_{N_1} \times LM_{N_2} \times \cdots \times LM_{N_t} \to LM_N \ .
\end{displaymath} 
Also, for each subset $K \subset N$, we get a forgetful map $\pi_K: LM_N
\to LM_K$, which is a toric morphism since it is a composition of maps
from Theorem~\ref{thm:univcurve}.

Recall there is a set of line bundles $\mathbb{G}_N$ on $LM_N$ \cite[Definition 1.5]{CT}, and one generates a larger set $\mathsf{H}_N$ of sheaves via
\begin{displaymath}
 \mathsf{H}_N := \left \lbrace \left(i_{N_1,\ldots,N_t}\right)_*
(G_{l_1} \boxtimes \ldots \boxtimes G_{l_t})
\mid \forall N_1 \cup \cdots \cup N_t = N \ , \ G_{l_j}
\in \mathbb{G}_{N_j} \right \rbrace,
\end{displaymath}
where $i_{N_1,\ldots,N_t}: Z_{N_1, \ldots, N_t} \hookrightarrow LM_N$ is the inclusion. 

\begin{thm}\label{thm:CTexcp}
 There is an ordering on the set 
 \begin{displaymath}
  \mathsf{CT}_N :=
\mathsf{H}_N \cup \left(
\bigcup_{K \subsetneq N} \{ \pi_K^\ast E \mid \ E \in \mathsf{H}_K \} \right)
\cup \{ \mathcal{O} \}
 \end{displaymath}
making it into an $(S_N \rtimes C_2)$-stable exceptional collection under permutations of the two sets of markings.
\end{thm}

\begin{proof}
 This is \cite[Proposition 1.5]{CT}.
\end{proof}

\begin{prop}\label{prop:CTaut=rootAut}
 The action of $S_{n+1} \rtimes C_2$ given by permuting the two sets of marked points corresponds to the action of $\operatorname{Aut}(A_n)$ on $X(A_n)$. 
\end{prop}

\begin{proof}
 We use the standard presentation of the root system for $A_n$ as $e_i - e_j$ for $1 \leq i < j \leq n+1$ and follow \cite[Construction 3.6]{BatyrevBlume}. The embedding $A_n \hookrightarrow A_{n+1}$ gives the universal curve $X(A_{n+1}) \to X(A_n)$. For $i \in \{1,\ldots,n\}$, we take the $(n+1)$ projections $A_{n+1} \to A_n$, whose kernels are generated by $e_i - e_{n+1}$ for $1 \leq i \leq n+1$. These give sections $s_i : X(A_n) \to X(A_{n+1})$. Finally, for the polar sections, we have the dual vector $v_{n+2}$. The vectors $v_{n+2}$ and $-v_{n+2}$ give toric invariant divisors which are isomorphic to $X(A_n)$ \cite[Proposition 1.9]{BatyrevBlume}. The isomorphisms give the other sections $s_0$ and $s_{\infty}$. 
 
 The Weyl group is the permutation group of the $e_i$, and hence of the $e_i - e_{n+2}$. In particular, it permutes the $s_i$. The outer involution acts on the fan by negation and thus exchanges the cone corresponding to $v_{n+2}$ with the cone corresponding to $-v_{n+2}$.
\end{proof}

\begin{cor}\label{cor:CTisstable}
 The set $\mathsf{CT}_N$ is $\operatorname{Aut}(\Sigma(A_n))$-stable. 
\end{cor}

\begin{proof}
 This is an immediate corollary of Lemma~\ref{lem:X(R)FanAut} and Proposition~\ref{prop:CTaut=rootAut}. 
\end{proof}

\begin{prop}\label{prop:CTistoric}
Each object in the collection $\mathsf{CT}_N$ is torus-equivariant. 
\end{prop}

\begin{proof}
Line bundles are always isomorphic to torus-equivariant line bundles,
so all objects in $\mathbb{G}_N$ are torus-equivariant.
There is a canonical equivariant structure on tensor products and on
pullbacks by equivariant morphisms (see \S{2}~of~\cite{BFK2});
thus each object $G_1 \boxtimes \ldots \boxtimes G_n$
is torus-equivariant for $G_{l_j} \in \mathbb{G}_{N_j}$.
Let $i : Z \to X$ be shorthand for some map $i_{N_1,\ldots,N_t}$.
There is a splitting of tori $T = S \times S'$ where $Z$ is an $S$-toric
variety and $S'$ acts trivially on $i(Z)$.
Let $\psi : T \to S$ denote the projection. 
We have a composition of functors
\[
\mathsf{D^b}(\operatorname{Coh}_S Z) \to
\mathsf{D^b}(\operatorname{Coh}_T Z) \to
\mathsf{D^b}(\operatorname{Coh}_T X)
\]
where the first map is the functor $\operatorname{Res}_\psi$
(\S{2.9}~of~\cite{BFK2}) and the second map is the $T$-equivariant
pushforward (\S{2.5}~of~\cite{BFK2}).
This composition reduces to the ordinary pushforward
$i_\ast : \mathsf{D^b}(Z) \to \mathsf{D^b}(X)$ when the equivariant
structure is forgotten.  We conclude that each object of
$\mathsf{H}_K$ is torus-equivariant
and the result follows.
\end{proof}

We now prove the main result of this section:

\begin{prop}\label{prop:X(An)excpcoll}
 Let $k$ be a field of characteristic zero and $X$ a form of $X(A_n)$ over $k$. Then $X$ admits a full exceptional collection of sheaves. 
\end{prop}

\begin{proof}
 Combining Theorem~\ref{thm:CTexcp}, Corollary~\ref{cor:CTisstable}, and
Proposition~\ref{prop:CTistoric} allows us to appeal to
Proposition~\ref{prop:blackbox} and conclude that $\mathsf{CT}_N$
descends to an exceptional collection of sheaves on $X$. 
\end{proof}

\begin{rem}
 To remove the characteristic zero assumption one needs to extend generation results of \cite{CT} to nonzero characteristic. This could conceivably be done by reversing the flow of reasoning in \cite{CT}, using the fact we know the collections for $V_n$ in any characteristic. We do not pursue this. 
\end{rem}


\bibliographystyle{alpha}
\bibliography{ExcCollToric}

\end{document}